%% file: main.tex
\documentclass{article}
\usepackage{graphicx} 
\usepackage{subcaption}
\usepackage[labelfont=bf, font=small, textfont=it, labelsep=colon]{caption}

\usepackage{comment}
\usepackage{amsmath}
\usepackage{amsfonts}
\usepackage{mathtools}
\usepackage{amssymb}
\usepackage{amsthm}

\newtheoremstyle{spaceddefinition}
  {3pt} 
  {6pt} 
  {}     
  {}     
  {\bfseries} 
  {.}    
  { }    
  {}     

\theoremstyle{spaceddefinition}

\newtheorem{theorem}{Theorem}[section]
\newtheorem{corollary}[theorem]{Corollary}
\newtheorem{conjecture}{Conjecture}
\newtheorem{lemma}[theorem]{Lemma}
\newtheorem{prop}[theorem]{Proposition}

\newtheorem{examp}{Example}[section]
\newtheorem{remark}{Remark}
\newtheorem{definition}{Definition}

\usepackage{hyperref}
\usepackage[top=2cm,bottom=2cm,left=3cm,right=3cm]{geometry}
\usepackage{parskip}
\input{preamble.tex}
\usepackage{tikz}
\usetikzlibrary{calc}
\usetikzlibrary{cd}
\usetikzlibrary{mindmap,backgrounds}
\usepackage[
  backend=biber,
  style=numeric,
  sorting=none,   
]{biblatex}

\title{On Conservative Matrix Fields: Continuous Asymptotics and Arithmetic}
\author{Shachar Weinbaum\thanks{Corresponding author: shacharwein@gmail.com} \and Elyasheev Leibtag \and Rotem Kalisch \and Michael Shalyt \and Ido Kaminer\thanks{Corresponding author: kaminer@technion.ac.il}}
\date{July 2026}

\begin{document}
\maketitle
\begin{abstract}
We present the Conservative Matrix Field (CMF) as a tool for the analysis and computation of D-finite functions. We use conservative matrix fields to establish asymptotic properties of families of linear forms in periods, including (but not limited to) multivariate Mellin integrals, via a discrete Levinson-type framework due to Benzaid and Lutz. Finally, we present an experimental analysis of the families of linear forms generated by these objects and formalize the resulting observations as conjectures on their continuous asymptotic and arithmetic properties.
\end{abstract}
\section{Introduction}
D-finite sequences, also known as P-recursive or holonomic sequences, are fundamental in many areas of mathematics, including combinatorics \cite{Stanley1999,FlajoletSedgewick2009}, transcendence \cite{Andre1989,Baker1975}, identity proofs \cite{PetkovsekWilfZeilberger1996}, and the theory of D-modules \cite{Kashiwara2003}. In 1979, Roger Ap\'ery found a ratio of two such sequences that converges to $\zeta(3)$ rapidly enough to prove its irrationality \cite{Apery1979}. His proof sparked a flurry of activity \cite{Beukers1979,Zudilin2003Catalan,Zudilin2022Zeta5,Calegari2024,Aptekarev2009,ZeilbergerZudilin2020} and motivated both deeper analysis of such ratios and experimental study of the sequences that produce them \cite{Raayoni2021,Razon2023,Elimelech2024,ChamberlandStraub2021,DoughertyBliss2021}. The tools generated to analyze such D-finite linear forms usually use information beyond the difference equations they satisfy. For example, in the analysis of linear forms coming from Mellin integrals, one uses the saddle-point method to establish asymptotic information. It would be beneficial to be able to analyze the quality of D-finite linear forms, by only analyzing the system of difference equations they satisfy.

In this paper, we present structural theorems and definitions for the Conservative Matrix Field (CMF), sometimes referred to in the literature as a discrete flat connection, a path-invariant matrix system \cite{Gosper1990}, a family of contiguity matrices \cite{MatsubaraHeoTelen2023}, or a difference Pfaffian system\cite{OharaTakayama2015}. In recent years, families of CMFs have been used to generalize rational approximations to fundamental constants \cite{Elimelech2024}, to construct a structured proof of the irrationality of Ap\'ery's constant \cite{David2023}, and to unify hundreds of formulas for $\pi$ \cite{RazUnifying2025}. We show this expressive power is since CMFs encode systems of difference equations. We provide a systematic way of getting asymptotic information from them, without invoking context dependent methods like saddle-point. We also include experimental results and powerful conjectures, that further the case for a recurrence-first approach.

Conservative matrix fields are collections of invertible skew-commutative matrices with rational-function entries (see Definition~\ref{def:CMF}). For example, consider the following pair of matrices, used in \cite{RazUnifying2025} to unify formulas of $\pi$:
\[
M_{1}(x_1,x_2) = \begin{pmatrix}
    0 & -(2x_1+1)x_1\\
    1 & 3x_1+x_2+2
\end{pmatrix},
\qquad
M_{2}(x_1,x_2) = \begin{pmatrix}
    x_2-x_1 & -(2x_1+1)x_1\\
    1 & 2x_1+2x_2+1
\end{pmatrix}.
\]
A direct calculation shows the skew-commutativity relation (writing $\textbf{x} = (x_1,x_2)$)
\[
M_{1}(\textbf{x})M_{2}(\textbf{x}+e_1) = M_{2}(\textbf{x})M_{1}(\textbf{x}+e_2).
\]
The Conservative Matrix Field (CMF) in this case, is a collection of matrices $\clM_\textbf{v}(\textbf{x})\in\GL_2(\Q)$, for $\textbf{v}\in \Z^2$ satisfying the conditions:
\[
    \clM_{\textbf{v}+\textbf{w}}(\textbf{x}) = \clM_\textbf{v}(\textbf{x})\cdot\clM_\textbf{w}(\textbf{x}+\textbf{v}) \qquad \text{and}\qquad \clM_{e_i}(\textbf{x}) = M_i(\textbf{x})
\]
This and other useful preliminary concepts are presented in Section~\ref{sec:Preliminaries}.

In Section~\ref{sec:D-fin_CMF}, we relate CMFs to D-finite functions and thereby connect this highly general framework to the more familiar setting of multivariate Mellin integrals (see Example~\ref{ex:BeukersZeta2}).

In Sections~\ref{sec:CMF_Asymp} and \ref{sec:CMF_Ratios}, we focus on applying CMFs to the analysis of families of P-recursive linear forms. These linear forms involve constants often called Ap\'ery limits, including zeta values \cite{Zudilin2022Zeta5,Fischler2004}, multiple zeta values \cite{Brown2009}, Euler's constant $\gamma$ \cite{Aptekarev2009}, and Catalan's constant $G$ \cite{Zudilin2003Catalan}. In Sections~\ref{sec:CMF_Asymp} and \ref{sec:CMF_Ratios}, we use CMFs together with a discrete Levinson-type framework due to Benzaid and Lutz \cite{BenzaidLutz1987} to characterize the asymptotic properties of such families of linear forms. This approach applies to a broad family of P-recursive linear forms, well beyond the setting of multivariate Mellin integrals.
The resulting asymptotic characterization takes the form of a factorization of the solution matrix of the system. Namely, for a CMF $\clM$ satisfying some properties, and for a suitable $\textbf{v}\in\Z^d$, one has
\begin{equation}\label{eq:IntroAsympFactor}
    \clM_{n\textbf{v}}(\textbf{x}) = B\diag(\lambda_i^n n^{\gamma_i})\,(I+o(1))\,A \qquad (n\to \infty)
\end{equation}
For some invertible matrices $A,B$, and constants $\lambda_i, \gamma_i$. The matrix $B$ encodes the Ap\'ery limits involved in the linear forms. This factorization serves as a computationally efficient and general alternative to the Laplace method often used for multivariate integrals. It also reflects the guiding principle of this paper: put the recurrences first, and thereby systematize the search for good linear forms.

Following the theoretical results, Section~\ref{sec:Experiments} presents an experimental analysis of CMFs and the linear forms they generate. The experiments exhibit a Stokes phenomenon, which we formalize as a conjecture on the structure of the asymptotic factorization in \eqref{eq:IntroAsympFactor}; this conjecture suggests an asymptotic analog of simultaneous diagonalization for skew-commutative matrices.
Section~\ref{sec:Experiments} also studies the arithmetic of these Diophantine approximations. We observe that the quality of the approximations, measured using naive height, appears to vary continuously as a function of the initial parameters defining the linear forms. This observation opens the door to optimization tools, such as simulated annealing, for the search for good linear forms. The continuity of the irrationality measure is formalized as a conjecture as well.
\clearpage
\section{Preliminaries}\label{sec:Preliminaries}
While there are more comprehensive sources on Ore algebras and D-finite functions \cite{Kauers2023,ChyzakSalvy1998} and on Ap\'ery limits \cite{DoughertyBliss2021,ChamberlandStraub2021,KoutschanZudilin2021}, the following subsections collect the definitions, theorems, and notation used later in the paper. The final subsection on conservative matrix fields contains similarly foundational definitions and results.
\subsection{Ore Algebra}
We recall the definition of an Ore algebra from \cite{Kauers2023}.
\begin{definition}\label{def:ore_algebra}
    Let $A$ be an integral domain, let $\sigma$ be an endomorphism of $A$, and let $\delta$ be a $\sigma$-derivation; that is, $\delta:A\to A$ is linear and satisfies
    \[
        \delta(ab) = \delta(a)b + \sigma(a)\delta(b)
    \]
    for all $a,b\in A$. Define $A[\partial]$ to be the ring of polynomials in $\partial$ with noncommutative multiplication determined by
    \[
        \partial a = \sigma(a)\partial + \delta(a), \qquad a\in A.
    \]
    Then $\A\coloneqq (A[\partial],\sigma,\delta)$ is an \textit{Ore algebra} over $A$.
\end{definition}

\begin{examp}\label{ex:univar_ore_algs}
While the generality of Definition~\ref{def:ore_algebra} is useful, this work mainly uses the following two Ore algebras. Let $K$ be a field of characteristic $0$.
\begin{itemize}
    \item $(K(z)[\theta_z],\Id,z\frac{d}{dz})$, the algebra of Euler differential operators;
    \item $(K(x)[S],\sigma,0)$, where $\sigma(x)=x+1$, the algebra of recurrence operators, satisfying $Sf(x)=f(x+1)S$.
\end{itemize}
\end{examp}
Ore algebras also admit multivariate versions under suitable compatibility conditions.
\begin{definition}\label{def:multivariate_ore_algebra}
For $1\le i\le d$, let $\A_i \coloneqq (A[\partial_i],\sigma_i,\delta_i)$ be Ore algebras, and assume that the $\sigma_i$ and $\delta_i$ commute with one another. One may then define the noncommutative ring $A[\partial_1,\dots,\partial_d]$ in which each $\partial_i$ satisfies the same commutation rule with elements of $A$ as in $\A_i$, and in addition
\[
    \partial_i\partial_j = \partial_j\partial_i.
\]
Then
\[
    \A\coloneqq (A[\partial_1,\dots,\partial_d],\sigma_1,\dots,\sigma_d,\delta_1,\dots,\delta_d)
\]
is a multivariate \textit{Ore algebra} over $A$.
\end{definition}
\begin{examp}
An example of such an extension is provided by Ore Laurent polynomials. For a field $K$ of characteristic $0$, consider
\[
    \A' \coloneqq (K(x)[S,S^{-1}],\sigma,\sigma^{-1},0,0).
\]
Since $\sigma$ commutes with its inverse and both derivations are zero, the conditions of Definition~\ref{def:multivariate_ore_algebra} are satisfied.
\end{examp}
We now introduce the multivariate Ore algebras that are central to the theory of conservative matrix fields.
\begin{remark}\label{rem:mult_ore_algebra_shift_notation}
    We will be interested in the following multivariate Ore algebras:
    \[
        \text{for } \textbf{x}=(x_1,\dots,x_d),\ \textbf{z}=(z_1,\dots,z_l),\qquad
        \A = K(\textbf{x},\textbf{z})[S_{x_1},\dots,S_{x_d},\theta_{z_1},\dots,\theta_{z_l}],
    \]
    \[
        \A' = K(\textbf{x},\textbf{z})[S_{x_1},S_{x_1}^{-1},\dots,S_{x_d},S_{x_d}^{-1},\theta_{z_1},\dots,\theta_{z_l}].
    \]
    Here $S_{x_i}$ is the shift operator in $x_i$, and $\theta_{z_i}$ is the Euler differential operator in $z_i$.

    For $\textbf{v}=(v_1,\dots,v_d)$, with $\textbf{v}\in\N^d$ in $\A$ and $\textbf{v}\in\Z^d$ in $\A'$, we write
    \[
        S_{\textbf{v}}\coloneqq \prod_{i=1}^d S_{x_i}^{v_i}
    \]
    and denote by $\sigma_{\textbf{v}}$ the corresponding composition of the $\sigma_i$. Thus, for $f\in K(\textbf{x},\textbf{z})$,
    \[
        S_{\textbf{v}}f(\textbf{x},\textbf{z}) = \sigma_{\textbf{v}}(f(\textbf{x},\textbf{z}))S_{\textbf{v}} = f(\textbf{x}+\textbf{v},\textbf{z})S_{\textbf{v}},
    \]
    and in particular $S_{\textbf{v}}S_{\textbf{w}} = S_{\textbf{v}+\textbf{w}}$.
\end{remark}
\subsection{D-finite Functions}
The ring of functions $f\in F\coloneqq K(\textbf{x})[[\textbf{z}]]$ forms a left module over the Ore algebras from Remark~\ref{rem:mult_ore_algebra_shift_notation}. We denote the action by a dot: for $L\in\A$ and $f\in F$, we write $L.f$. The set of operators that annihilate $f$ is called its annihilator ideal.
\begin{definition}
    The \textit{annihilator} of a function $f(\textbf{x},\textbf{z})\in F$ in the Ore algebra $\A = K(\textbf{x},\textbf{z})[S_{x_1},\dots,S_{x_d},\theta_{z_1},\dots,\theta_{z_l}]$ is the set
    \[
        \text{ann}(f) = \set{L\in \A : L.f = 0}.
    \]
\end{definition}
\begin{remark}
    $\text{ann}(f)\subset \A$ is a left ideal of the Ore algebra $\A$.
\end{remark}
It is also useful to consider the image of a function under an Ore algebra.
\begin{definition}
The image of a function $f\in F$ under an Ore algebra $\A$ is the set
\[
    \A.f \coloneqq \set{L.f : L\in\A}.
\]
\end{definition}
\begin{definition}\label{def:D_finite}
    A \textit{D-finite function} $f$ is a function for which:
    \[
        \dim_{K(\textbf{x},\textbf{z})} \A/\text{ann}(f)
        = \dim_{K(\textbf{x},\textbf{z})} \A.f
        < \infty.
    \]
\end{definition}
\begin{remark}\label{rem:Dfin_other_defs_and_neg_shift} ${}$
\begin{itemize}
    \item The condition in Definition~\ref{def:D_finite} is equivalent to either $\dim \operatorname{span}\set{\partial_i^k.f : k\in\N}<\infty$ for all $i$, or to $K[\partial_i]\cap \text{ann}(f)\neq\emptyset$ for all $i$.
    \item For any $f\in F$, we have $S_{x_i}.f\equiv 0$ if and only if $f\equiv 0$. Hence, if $f\in F$ is D-finite and $\sum_{j=0}^r p_j S_{x_i}^j$ is the minimal-degree operator in $K[S_{x_i}]\cap \text{ann}(f)$, then $p_j\in K(\textbf{x},\textbf{z})$ and $p_0,p_r\not\equiv 0$.
    \item For such a minimal operator,
    \[
        \left(\sum_{j=0}^r p_jS_{x_i}^j\right).f = 0
        \implies
        \left(-\sum_{j=1}^r \frac{p_j}{p_0}S_{x_i}^{j}\right).f = f,
    \]
    and therefore
    \[
        \left(-\sum_{j=1}^r \sigma_i^{-1}\left(\frac{p_j}{p_0}\right)S_{x_i}^{j-1}\right).f
        = \sigma_i^{-1}(f) = S_{x_i}^{-1}.f.
    \]
    \item Consequently, for $f\in F$ D-finite over $\A$, one has $\A.f = \A'.f$.
\end{itemize}
\end{remark}
\begin{examp}
Examples of D-finite functions include:
\begin{itemize}
    \item The binomial coefficients $f(n,k)=\binom{n}{k}$ are D-finite in $\A = \Q(n,k)[S_n,S_k]$.
    \item Bessel functions such as $J_n(z)$ and $I_n(z)$ are D-finite in $\A = \Q(n,z)[S_n,\theta_z]$.
    \item Generalized hypergeometric functions $\pFq{p}{q}{x_1,\dots,x_p}{x_{p+1},\dots,x_{p+q}}{z}$ are D-finite in $\A = \Q(\textbf{x},z)[S_{x_1},\dots,S_{x_{p+q}},\theta_z]$.
\end{itemize}
\end{examp}
For the special case of the univariate Ore algebra $K(n)[S_n]$ of recurrence operators, D-finite functions are called \textit{D-finite sequences}, or sometimes \textit{holonomic} or \textit{P-recursive} sequences.
\subsection{Ap\'ery Limits and the Irrationality Measure}
\begin{definition}\label{def:AperyLimit}
    Let $L\in \Q(n)[S_n]$ be a recurrence operator of order at least $2$, and let $u_1(n),u_2(n)\in\Q^{\N}$ be two solutions of $L$. If the ratio $\frac{u_1(n)}{u_2(n)}$ converges, its limit is called an \textit{Ap\'ery limit}, and the ratio sequence itself will be referred to in this paper as a \textit{D-finite ratio}.
\end{definition}
These D-finite ratios lie at the heart of many irrationality results \cite{Beukers1979,Zudilin2003Catalan,Zudilin2022Zeta5,Aptekarev2009,ZeilbergerZudilin2020}, and Ap\'ery limits include $\pi$, $\gamma$, $e$, Catalan's constant $G$, and many other constants.
In his seminal 1979 work, Roger Ap\'ery proved the irrationality of $\zeta(3)$ \cite{Apery1979} by constructing a D-finite ratio that converges to it and satisfies Dirichlet's criterion.
\begin{theorem}
    \label{thm:dirichlet_criterion}
   (Dirichlet, 1834) Let $l\in\R$ and $p_n,q_n\in\Z\setminus\set{0}$. If $\abs{\set{\frac{p_n}{q_n}:n\in\N}} = \infty$ and
   \[
       \exists_{C,\delta>0}: \abs{\frac{p_n}{q_n}-l}<\frac{C}{\abs{q_n}^{1+\delta}},
   \]
   then $l\notin\Q$.
\end{theorem}
\begin{remark}
    For $l\neq 0$, the same conclusion holds if one compares against the numerators instead:
    \[
        \exists_{C,\delta>0}: \abs{\frac{p_n}{q_n}-l}<\frac{C}{\abs{p_n}^{1+\delta}}.
    \]
\end{remark}
Proving irrationality in this manner requires sequences that balance rapid convergence with slow growth in their numerators and denominators. These notions require definitions:

\begin{definition}
    \label{def:convergence_rate}
For a sequence $s_n\in\Q$, such that $s_n\to l$, its \textit{convergence rate sequence} is defined as:
\[\rho_n \coloneqq \frac{\log\abs{s_n-l}}{n}\]
The limit of $\rho_n$, if it exists, will be called the \textit{convergence rate} of the sequence ($\rho \coloneqq \lim_{n\to\infty}\rho_n$).
A sequence is said to have \textit{trivial} convergence rate if $\rho = 0$.
\end{definition}
\begin{definition}
    \label{def:limit_height}
For a sequence $s_n\in\Q$, such that $s_n\to l$, its \textit{height sequence} is defined as: 
\[\eta_n = \frac{\log\abs{H(s_n)}}{n}\]
where $H(s_n)$ is the classical/naive height, meaning the maximum of the absolute value of the reduced numerator and denominator:
\[H(c) \coloneqq \max\set{\abs{a},\abs{b}} \text{ s.t } a,b\in\Z, \gcd(a,b) = 1, \frac{a}{b} = c.\]
The limit of $\eta_n$, if it exists, will be called the \textit{height} of the sequence ($\eta = \lim_{n\to\infty}\eta_n$). A sequence is said to have \textit{trivial} height if $\eta = 0$.
\end{definition}
\begin{definition}
    \label{def:irrationality_measure_of_seq}
For a sequence $s_n\in\Q$, such that $s_n\to l$, its \textit{irrationality measure sequence} is the sequence of real numbers $\delta_n$ satisfying:
\[\abs{s_n-l} = \frac{1}{H(s_n)^{1+\delta_n}},\]
The limit of $\delta_n$, if it exists, will be called the \textit{irrationality measure} of the sequence ($\delta \coloneqq \lim_{n\to\infty}\delta_n$). A sequence is said to have \textit{trivial} irrationality measure if $\delta = -1$.
\end{definition}
\begin{remark}\label{rem:limsup_delta}
While the existence of the limit of $\delta_n$ is often easy to verify numerically, it is frequently hard to prove. However, one may still use $\delta_n$ to invoke Dirichlet's criterion:
\[\lim\sup \delta_n>0\implies l\notin\Q\]
(assuming also $s_n$ does not repeat any element infinitely many times). 
\end{remark}
\begin{remark}\label{rem:convergence_height_delta_connection}
    These three quantities are related by
    $$\delta_n = -1-\frac{\rho_n}{\eta_n}$$
    Thus:
    \begin{itemize}
        \item If any two of the sequences $\set{\delta_n, \rho_n,\eta_n}$ have finite nontrivial limits, then the third sequence converges as well.
        \item If all three sequences converge for $s_n$, then the subsequence $s'_n \coloneqq s_{kn}$ (for $k\in\N\setminus\set{0}$) has the same limit $l$ and the same irrationality measure $\delta$, while $\rho' = k\rho$ and $\eta' = k\eta$.
        \item If the sequence $s_n$ has nontrivial convergence rate and $\liminf\eta_n<-\rho$, then by Remark~\ref{rem:limsup_delta} one has $l\notin\Q$.
    \end{itemize}
\end{remark}
The final point in Remark~\ref{rem:convergence_height_delta_connection} is the key ingredient in Ap\'ery's proof of the irrationality of $\zeta(3)$; proofs may be found in \cite{Apery1979,Beukers1979,David2023}.
\begin{theorem}\label{thm:aperys_theorem}
    Consider the following D-finite equation:
    $$(n+2)^3u(n+2)-(2n+3)(17n^2+51n+39)u(n+1) +(n+1)^3 u(n) = 0$$
and two solutions, $u_1(n), u_2(n)$ satisfying the following initial conditions:
$$u_1(0) = 0, u_1(1) = 6, u_2(0) = 1, u_2(1) = 5$$
Then the sequence $s_n = \frac{u_1(n)}{u_2(n)}$ converges to $\zeta(3)$, with convergence rate $\rho = -8\log(\sqrt{2}+1)$ and height satisfying $\limsup\eta_n\le 4\log(\sqrt{2}+1)+3$. Consequently,
\[
    \liminf\delta_n \ge \frac{4\log(\sqrt{2}+1)-3}{4\log(\sqrt{2}+1)+3}\approx 0.08\dots,
\]
which proves that $\zeta(3)\notin\Q$.
\end{theorem}
\subsection{The Conservative Matrix Field (CMF)}
In this and the following sections, we use the notation $S_{\textbf{v}}$ and $\sigma_{\textbf{v}}$ from Remark~\ref{rem:mult_ore_algebra_shift_notation} for $\textbf{v}\in\Z^d$. We also extend the Ore algebra to matrices of rational functions by applying the operators entrywise: for $M\in M_{r\times r}(K(\textbf{x},\textbf{z}))$, say $M=(m_{i,j})$, we set
$$\sigma(M) \coloneqq (\sigma(m_{i,j})),\delta(M) = (\delta(m_{i,j})), \partial M = \sigma(M)\partial+\delta(M).$$
Let us begin with the definition of a conservative matrix field:
\begin{definition}\label{def:CMF}
A \textit{Conservative Matrix Field} (abbreviated CMF) of dimension $d$ and rank $r$ over the field $K$ is a map
\[\clM: \Z^d \to \text{GL}_r(K(\mathbf{x})), \quad \textbf{v} \mapsto \clM_\textbf{v}\]
with the property that 
\begin{equation} \label{eq:cocycle}
 \clM_{\textbf{v}+\textbf{w}} = \clM_\textbf{v} \cdot \sigma_\textbf{v}(\clM_\textbf{w})
\end{equation} 
or equivalently, as matrices of Ore operators:
\begin{equation} \label{eq:Ore_cocycle}
 \clM_{\textbf{v}+\textbf{w}}S_{\textbf{v}+\textbf{w}} = \clM_\textbf{v} S_\textbf{v} \clM_\textbf{w}S_\textbf{w}
\end{equation} 
for any $\textbf{v},\textbf{w} \in \Z^d$. We denote the set of all such conservative matrix fields by $\clZ^1(\Z^d,\text{GL}_r(K(\mathbf{x})))$.
\end{definition}
Equation~\eqref{eq:cocycle} is known as the \textit{cocycle equation}\footnote{In fact, a conservative matrix field is an Eilenberg--MacLane 1-cocycle of $\Z^d$ with non-abelian coefficients $\text{GL}_r(K(\mathbf{x}))$.}.
For any conservative matrix field $\clM$ and any $\textbf{v}\in \Z^d$, we have $\clM_0 = \Id_r$ and $\clM_{-\textbf{v}} = (\sigma_{-\textbf{v}}(\clM_\textbf{v}))^{-1}$. 
Now, let $e_1,\ldots,e_d$ be the standard $\Z$-basis of $\Z^d$. If $\clM$ is a conservative matrix field, then, by virtue of the cocycle equation \eqref{eq:cocycle}, we have the equality 
$$
    \clM_{e_i} \cdot \sigma_i(\clM_{e_j}) = \clM_{e_i+e_j} = \clM_{e_j} \cdot \sigma_j(\clM_{e_i})
$$
for every $i,j \in \set{1,\ldots,d}$. This leads to a more practical equivalent definition of conservative matrix fields, stated in the following proposition.
\begin{prop} \label{prop:generators}
Let $M_1,\ldots,M_d \in \text{GL}_r(K(\mathbf{x}))$ be $d$ matrices that satisfy the equations
\begin{equation} \label{eq:cond_generators}
	M_i \cdot \sigma_i(M_j )= M_j \cdot \sigma_j(M_i) \quad \mbox{for all } i,j \in \set{1,\ldots,d}. 
\end{equation}
Then there exists a unique $d$-dimensional conservative matrix field $\clM:\Z^d \to \text{GL}_r(K(\mathbf{x}))$ such that
$\clM_{e_i} = M_i$
for every $i \in \set{1,\ldots,d}$. 
\end{prop}
\begin{proof}
Using \eqref{eq:cocycle}, one can uniquely define a CMF $\clM$ satisfying the above by setting
$$
\clM_{\textbf{v}} = \clM_{\sum a_i e_i}
= \clM_{a_1 e_1}\cdot \sigma_{a_1e_1}\bigl(\clM_{a_2 e_2}\cdot \sigma_{a_2e_2}(\cdots \clM_{a_d e_d})\dots\bigr).
$$
This is well defined because one may use \eqref{eq:cond_generators} inductively to show that \eqref{eq:cocycle} holds in general.
\end{proof}
A few examples are in order:
\begin{examp}\label{ex:Zeta3CMF}
Consider the following matrices $M_1,M_2 \in \text{GL}_2(\Q(\mathbf{x}))$:
$$M_{1} =  \begin{pmatrix}0 & -1\\\frac{\left(x_{1} + 1\right)^{3}}{x_{1}^{3}} & \frac{x_{1}^{3} + 2 x_{2} \left(2 x_{1} + 1\right) \left(x_{2} - 1\right) + \left(x_{1} + 1\right)^{3}}{x_{1}^{3}}\end{pmatrix} \quad M_{2}= \begin{pmatrix}\frac{- x_{1}^{3} + 2 x_{1}^{2} x_{2} - 2 x_{1} x_{2}^{2} + x_{2}^{3}}{x_{2}^{3}} & - \frac{x_{1}^{3}}{x_{2}^{3}}\\\frac{x_{1}^{3}}{x_{2}^{3}} & \frac{x_{1}^{3} + 2 x_{1}^{2} x_{2} + 2 x_{1} x_{2}^{2} + x_{2}^{3}}{x_{2}^{3}}\end{pmatrix}$$
These matrices satisfy $M_1\sigma_{1}(M_2) = M_2\sigma_{2}(M_1)$, and thus from Proposition \ref{prop:generators} they generate a conservative matrix field $\clM$, of dimension $2$, rank $2$, over field $\Q$. For example:
\tiny
$$\clM_{(2,0)} = M_1\sigma_{1}(M_1) = 
\begin{pmatrix}
 -\frac{\left(x_1+2\right){}^3}{\left(x_1+1\right){}^3} & -\frac{\left(2 x_1+3\right) \left(x_1 \left(x_1+3\right)+2 \left(x_2-1\right) x_2+3\right)}{\left(x_1+1\right){}^3} \\
 \frac{\left(x_1+2\right){}^3 \left(2 x_1+1\right) \left(x_1^2+x_1+2 \left(x_2-1\right) x_2+1\right)}{x_1^3 \left(x_1+1\right){}^3} & \frac{\left(2 x_1+1\right) \left(2 x_1+3\right) \left(x_1^2+x_1+2 \left(x_2-1\right) x_2+1\right) \left(x_1 \left(x_1+3\right)+2 \left(x_2-1\right) x_2+3\right)-\left(x_1+1\right){}^6}{x_1^3 \left(x_1+1\right){}^3} \\
\end{pmatrix}$$
\normalsize
$$\clM_{(0,-1)} = \sigma_{(0,-1)}(M_2)^{-1} = 
\begin{pmatrix}
 \frac{\left(x_1^2+\left(x_2-1\right) x_1+\left(x_2-1\right){}^2\right) \left(x_1+x_2-1\right)}{\left(x_2-1\right){}^3} & \frac{x_1^3}{\left(x_2-1\right){}^3} \\
 -\frac{x_1^3}{\left(x_2-1\right){}^3} & -\frac{\left(x_1^2-\left(x_2-1\right) x_1+\left(x_2-1\right){}^2\right) \left(x_1-x_2+1\right)}{\left(x_2-1\right){}^3} \\
\end{pmatrix}$$
\end{examp}
\begin{examp}\label{ex:2f1CMF_first_example}
    Consider the following matrices $M_1,M_2,M_3 \in \text{GL}_2(\Q(z)(\mathbf{x}))$: 
    $$M_1 = 
\begin{pmatrix}
 1 & \frac{x_2 z}{1-z} \\
 \frac{1}{x_1} & \frac{x_1+x_2 z-x_3+1}{x_1(1-z)} \\
\end{pmatrix}\quad M_2 = \begin{pmatrix}
 1 & \frac{x_1 z}{1-z} \\
 \frac{1}{x_2} & \frac{x_1 z+x_2-x_3+1}{x_2(1-z)} \\
\end{pmatrix} $$$$M_3 = \frac{x_3}{(x_1-x_3)(x_2-x_3)}\begin{pmatrix}
 -x_1-x_2+x_3 & x_1 x_2 \\
 \frac{1}{z}-1 & \frac{x_3 (z-1)}{z} \\
\end{pmatrix} $$
These satisfy the condition in Equation \eqref{eq:cond_generators}, and hence generate a conservative matrix field $\clM$ of dimension $3$, rank $2$, over field $\Q(z)$. In addition, for any value $z_0\in\C\setminus\set{0,1}$, substituting $z=z_0$ in $M_1,M_2,M_3$ yields a CMF of dimension $3$, rank $2$, over $\C$.
\end{examp}
\begin{examp}\label{ex:1dCMF}
Any matrix $M_1\in\text{GL}_r(K(x))$ generates a conservative matrix field of dimension $1$, rank $r$, as it immediately satisfies the condition in Equation \eqref{eq:cond_generators}.
\end{examp}
\begin{examp}\label{ex:Const3x3}
Consider the following matrices $M_1,M_2\in \text{GL}_3(\Q(\mathbf{x}))$:
$$M_{1} =\begin{pmatrix}-9 & 2 & 2\\-38 & 11 & 4\\-24 & 4 & 7\end{pmatrix}\quad M_{2} = \begin{pmatrix}\frac{119}{6} & - \frac{7}{6} & - \frac{37}{6}\\\frac{343}{6} & - \frac{11}{6} & - \frac{119}{6}\\39 & - \frac{7}{3} & -12\end{pmatrix}$$
These matrices are members of the subset $\text{GL}_3(\Q)$, and they commute. This implies:
$$M_1 \sigma_{1}(M_2) = M_1M_2 = M_2M_1 =M_2\sigma_{2}(M_1) $$
Hence the condition in \eqref{eq:cond_generators} is satisfied, and these matrices generate a CMF $\clM$ of dimension $2$ and rank $3$ over $\Q$. In this case, it is easy to calculate the closed form of $\clM$:
$$\clM_{(a,b)} = M_1^aM_2^b$$
\end{examp}
\begin{remark}\label{rem:MultiHyperTerm}
    ${}$
    \begin{itemize}
        \item The determinant of a CMF $\clM$ of dimension $d$, rank $r$ over $K$, is a CMF of dimension $d$, rank $1$ over $K$
        \item a CMF of dimension $d$, rank $1$ over $K$ is also known as a Multivariate Hypergeometric Term \cite{WilfZeilberger1992,Payne2014}.
    \end{itemize}
\end{remark}
Two more useful notions are of an \textit{evaluation} of a CMF, and of a \textit{trajectory matrix}:
\begin{definition}
For a CMF $\clM$, consider the composition with the evaluation map (where it is defined). for $\textbf{x}\in K^d\ , \ \textbf{v}\in\Z^d$
$$\text{ev}_\textbf{x}\circ \clM :\Z^d\to \text{M}_{r\times r}(K)\ ,\ \textbf{v}\mapsto \clM_\textbf{v}(\textbf{x})$$
Notice that $\clM_\textbf{v}(\textbf{x})$ can be singular, or not defined, if one of the entries of $\clM_\textbf{v}$ is not defined at $\textbf{x}$, or in case $\det \clM_\textbf{v}(\textbf{x})=0$.
\end{definition}
\begin{examp}
Using the CMF defined in Example~\ref{ex:Zeta3CMF}, we note that:
    $$\clM_{(2,0)}(1,1) = 
\begin{pmatrix}
 -\frac{27}{8} & -\frac{35}{8} \\
 \frac{243}{8} & \frac{251}{8} \\
\end{pmatrix}\quad \clM_{(2,0)}(-2,0) = 
\begin{pmatrix}
 0 & -1 \\
0 & 1\\
\end{pmatrix}\notin\text{GL}_2(\Q)$$
However, $\clM_{(0,-1)}(1,1)$ is not defined.
\end{examp}
Notice that from \eqref{eq:cocycle} we have that this is a multiplicative function:
$\sigma_\textbf{v}(\clM_\textbf{w})(\textbf{x}) = \clM_\textbf{w}(\textbf{x}+\textbf{v})$ giving:
\begin{equation}\label{eq:path_indep}
    \clM_{\textbf{v}+\textbf{w}}(\textbf{x}) = \clM_\textbf{v}(\textbf{x})\cdot\clM_\textbf{w}(\textbf{x}+\textbf{v})
\end{equation}
Geometrically, this map ($\clM_\textbf{v}(\textbf{x})$) associates a matrix with the translation from point $\textbf{x}$ to point $\textbf{x}+\textbf{v}$. This, as demonstrated in \eqref{eq:path_indep}, in a \textbf{path-independent} manner (See Figure \ref{fig:Comm_Diagram}). Hence, the resemblance to conservative vector fields, and a justification for the used terminology. 
\begin{figure}
    \centering
\begin{tikzpicture}
    \def \n {4}
    \def \gapv {0.3cm}
    \def \gaph {0.5cm}
    \def \distance {2.6cm}
    
    \draw[->,violet] (1*\distance+\gapv, 2*\distance+\gapv) -- (4*\distance-\gapv, 3*\distance-\gapv) node[midway, below, sloped] {\footnotesize $\clM_{(3,1)}(1,2)$};
    \draw[->,teal] (1*\distance+\gapv, 3*\distance+\gapv) -- (2*\distance-\gapv, 4*\distance-\gapv) node[midway, below, sloped] {\footnotesize $\clM_{(1,1)}(1,3)$};
    \draw[->,orange] (3*\distance+\gapv, 2*\distance-\gapv) -- (4*\distance-\gapv, 1*\distance+\gapv) node[midway, below, sloped] {\footnotesize $\clM_{(1,-1)}(3,2)$};

    \foreach \i in {1,...,\n} {
        \foreach \j in {1,...,\n} {
            \coordinate (p) at (\i*\distance, \j*\distance);

            \ifnum\i<\n
                \draw[->] (\i*\distance+\gaph, \j*\distance) -- ++(\distance-2*\gaph, 0) node[midway, below] {\footnotesize $\clM_{e_1}(\i,\j)$};
            \fi

            \ifnum\j<\n
                \draw[->] (\i*\distance, \j*\distance+\gapv) -- ++(0,\distance-2*\gapv) node[midway, rotate=90, anchor=center, above] {\footnotesize $\clM_{e_2}(\i,\j)$};
            \fi

            \node at (\i*\distance, \j*\distance) {(\i,\j)};
        }
    }


    \foreach \j in {1,...,\n} {
        \draw[dashed,->] (\n*\distance+\gaph, \j*\distance) -- ++(1cm, 0);
    }

    \foreach \j in {1,...,\n} {
        \draw[dashed,->] (1*\distance-\gaph, \j*\distance) -- ++(-1cm, 0);
    }

    \foreach \i in {1,...,\n} {
        \draw[dashed,->] (\i*\distance, \n*\distance+\gapv) -- ++(0, 1cm);
    }

    \foreach \i in {1,...,\n} {
        \draw[dashed,->] (\i*\distance, 1*\distance-\gapv) -- ++(0, -1cm);
    }

\end{tikzpicture}
    \caption{The geometric interpretation of the evaluated CMF $\clM_\textbf{v}(\textbf{x})$, depicted over a subset of the lattice $\Z^2$. The black arrows encode the translations up and to the right. The colored arrows encode a subset of other possible translations. This figure is a commutative diagram.}
    \label{fig:Comm_Diagram}
\end{figure}

It is useful to examine the matrices associated with straight-line paths in $\textbf{x}+\Z^d$. We call such paths, starting at $\textbf{x}$ and advancing by $\textbf{v}$ at each step, \textit{trajectories}, and we call the associated sequence of matrices \textit{trajectory matrices}. The formal definition is as follows.
\begin{definition}
    In a CMF $\clM$ of dimension $d$ and rank $r$ over $K$, the \textit{trajectory matrix} $T_{\textbf{x},\textbf{v}}(n)$ is the matrix associated with the $n$th step, starting at $0$, along the trajectory $\textbf{x}+n\textbf{v}$:
    $$T_{\textbf{x},\textbf{v}}(n) \coloneqq \clM_\textbf{v}(\textbf{x}+n\textbf{v})\in \text{GL}_r(K(n))$$
\end{definition}
Using \eqref{eq:path_indep}, we obtain the following.
\begin{remark}
    Define matrix-product notation by
    $$\prod_{k=0}^n M_k = M_0\cdot M_1\cdots M_n$$
    so that
    \begin{equation}\label{eq:prod_of_traj}
         \clM_{n\textbf{v}}(\textbf{x}) = \prod_{k=0}^{n-1} T_{\textbf{x},\textbf{v}}(k)
    \end{equation}
\end{remark}
\begin{examp}\label{ex:Zeta3Traj}
Using the CMF defined in Example \ref{ex:Zeta3CMF}, we see that:
$$T_{(1,1),(1,0)}(n) = \begin{pmatrix}0 & -1\\\frac{\left(n + 2\right)^{3}}{\left(n + 1\right)^{3}} & \frac{\left(2 n + 3\right) \left(n^{2} + 3 n + 3\right)}{\left(n + 1\right)^{3}}\end{pmatrix}$$ $$T_{(1,1),(1,1)}(n) = \begin{pmatrix}- \frac{\left(n + 2\right)^{3}}{\left(n + 1\right)^{3}} & \frac{\left(- 2 n - 3\right) \left(3 n^{2} + 9 n + 7\right)}{\left(n + 1\right)^{3}}\\\frac{6 \left(n + 2\right)^{3}}{\left(n + 1\right)^{3}} & \frac{35 n^{3} + 159 n^{2} + 243 n + 125}{\left(n + 1\right)^{3}}\end{pmatrix}$$
\end{examp}
\begin{remark}
Note that, by Example~\ref{ex:1dCMF}, trajectory matrices generate a one-dimensional CMF. Since one-dimensional CMFs are very well studied under the guise of D-finite systems, CMFs may be viewed as higher-dimensional generalizations of D-finite systems.
\end{remark}
A final preliminary notion that requires a definition is that of a \textit{coboundary transformation}:
\begin{definition}\label{def:coboundary}
    Two CMFs $\clM^1,\clM^2$ of dimension $d$ and rank $r$ over $K$ are \textit{coboundary equivalent}, which we denote by $\clM^1 \sim \clM^2$, if there exists a matrix $A \in \text{GL}_r(K(\mathbf{x}))$ such that
    \begin{equation}
         A \cdot \clM^1_\textbf{v} = \clM^2_\textbf{v} \cdot \sigma_\textbf{v}(A)
    \end{equation}
    or equivalently:
    \begin{equation}
         A \clM^1_\textbf{v} S_\textbf{v} = \clM^2_\textbf{v} S_\textbf{v} A
    \end{equation}
    for every $\textbf{v} \in \Z^d$.\\
The matrix $A$ will be referred to as the \textit{coboundary matrix}.
\end{definition}
It is a useful exercise to verify that a coboundary transformation $\clM_\textbf{v}\mapsto A^{-1}\cdot\clM_\textbf{v}\cdot \sigma_\textbf{v}(A)$ indeed satisfies the cocycle condition \eqref{eq:cocycle}. A few examples follow.

\begin{examp}\label{ex:Zeta3CoboundUnbalanced}
The CMF from Example \ref{ex:Zeta3CMF} is coboundary equivalent to the CMF generated by:
$$\overline{M}_1 = \begin{pmatrix}
\frac{x_{2} \left(2 x_{1} + 1\right) \left(x_{2} - 1\right) + \left(x_{1} + 1\right)^{3}}{x_{1}^{3}} & \frac{x_{2}^{2} \left(2 x_{1} + 1\right) \left(x_{2} - 1\right)}{x_{1}^{3}}\\\frac{x_{1}^{3} + x_{2} \left(2 x_{1} + 1\right) \left(x_{2} - 1\right) + \left(x_{1} + 1\right)^{3}}{x_{1}^{3} x_{2}} & \frac{x_{1}^{3} + x_{2} \left(2 x_{1} + 1\right) \left(x_{2} - 1\right)}{x_{1}^{3}}
\end{pmatrix}\quad \overline{M}_2 = \begin{pmatrix}
    \frac{2 x_{1}^{2}}{x_{2}^{2}} + 1 & \frac{2 x_{1} \left(x_{2} + 1\right)}{x_{2}}\\\frac{2 x_{1} \left(x_{1}^{2} + x_{2}^{2}\right)}{x_{2}^{4}} & \frac{\left(2 x_{1}^{2} + x_{2}^{2}\right) \left(x_{2} + 1\right)}{x_{2}^{3}}
\end{pmatrix}$$
Via the coboundary matrix $A = \begin{pmatrix}
    1 & -x_2\\
    1 & x_2
\end{pmatrix}$. Indeed,
$$\overline{M}_1 = A^{-1}M_1\sigma_{1}(A)\quad\overline{M}_2 = A^{-1}M_2\sigma_{2}(A)$$
\end{examp}
\begin{examp}
If we define $\overline{\clM} \coloneqq \sigma_\textbf{w}(\clM)$, then $\overline{\clM} \sim \clM$. Indeed, taking $A = \clM_\textbf{w}$, we have
$$A\overline{\clM}_\textbf{v} =  \clM_\textbf{w} \sigma_\textbf{w}(\clM_\textbf{v}) = \clM_{\textbf{v}+\textbf{w}} = \clM_\textbf{v}\sigma_\textbf{v}(\clM_\textbf{w}) =\clM_\textbf{v}\sigma_\textbf{v}(A) $$
\end{examp}
\begin{examp}\label{ex:Zeta3TrajCompanion}
Note how coboundary transformations affect trajectory matrices. Recall the trajectory matrices calculated in Example~\ref{ex:Zeta3Traj}. Using the coboundary matrix
$$A = \begin{pmatrix}
   \frac{1}{x_{1}^{3}} & - \frac{1}{x_{2}^{3}}\\0 & \frac{x_{1}^{3} + 2 x_{1}^{2} x_{2} + 2 x_{1} x_{2}^{2} + x_{2}^{3}}{x_{1}^{3} x_{2}^{3}}
\end{pmatrix}$$
on the CMF from Example~\ref{ex:Zeta3CMF}, one gets that the trajectory matrix of the coboundary CMF $\overline{\clM}$ is
$$\overline{T}_{(1,1),(1,1)}(n) = \begin{pmatrix}
    0 & - \frac{\left(n + 1\right)^{3}}{\left(n + 2\right)^{3}}\\1 & \frac{\left(2 n + 3\right) \left(17 n^{2} + 51 n + 39\right)}{\left(n + 2\right)^{3}}
\end{pmatrix}$$
Note that this matrix is in companion form, and encodes the recurrence:
$$(n+2)^3u(n+2)-(2n+3)(17n^2+51n+39)u(n+1) +(n+1)^3 u(n) = 0$$
This is the recurrence used in Ap\'ery's proof of the irrationality of $\zeta(3)$ (see Theorem~\ref{thm:aperys_theorem}).
\end{examp}
Finally, the notions of dual CMFs and sub-CMFs are described in Appendix~\ref{app:Dual_Sub_CMF}.
\clearpage
\section{CMFs and D-finite functions}\label{sec:D-fin_CMF}
Other discussions of conservative matrix fields \cite{David2023,Elimelech2024,RazUnifying2025,Gosper1990} highlight their utility in unifying approximations of constants, combinatorial identity proofs, and irrationality proofs. However, apart from the last reference, these works only discuss CMFs of rank $r=2$. Theorem \ref{thm:Dfin_CMF} shows how to construct nontrivial conservative matrix fields of general dimension and rank using D-finite functions. Lemma \ref{lem:MatrixCoeffsDFinite}  shows that the matrix coefficients of CMFs are themselves D-finite functions, and explain how this fact can be used for their computation.

Consider a D-finite function $f(\textbf{x},\textbf{z})\in F$. As discussed in Remark~\ref{rem:Dfin_other_defs_and_neg_shift}, the image of $f$ under the Ore algebra $\A'$ is finite dimensional; say $\dim_{K(\textbf{x},\textbf{z})} \A'.f = r$. Given a basis $B = (b_1.f,\dots,b_r.f)$, with $b_i\in\A'$, of $\A'.f$ over $K(\textbf{x},\textbf{z})$, every element of the image is a $K(\textbf{x},\textbf{z})$-linear combination of $(b_1.f,\dots,b_r.f)$.
\begin{remark}${}$
\begin{itemize}
    \item For $L\in \A'$, the coordinates of $L.f$ in the basis $B$ can be computed using an appropriate Gr\"obner basis.
    \item $\sigma_\textbf{v}(B)$ is also a basis of $\A'.f$, by the invertibility of the action of $\sigma_\textbf{v}$ on $F$.
\end{itemize}
\end{remark}
\begin{definition}\label{def:basis_change_matrix}
Consider a D-finite function $f(\textbf{x},\textbf{z})\in F$ and a basis $B = (b_1.f,\dots,b_r.f)$, with $b_i\in\A'$, of $\A'.f$.\\
The \textit{basis change matrix} $\clM^f_{\textbf{v}}$ is the matrix that maps $B$ to $S_{\textbf{v}}.B$; that is, it is the unique matrix $\clM^f_\textbf{v}\in \text{GL}_r(K(\textbf{x},\textbf{z}))$ satisfying
\begin{equation}
\label{eq:dfin_basis_change}
    (b_1.f,\dots,b_r.f)\cdot \clM^f_\textbf{v} =(S_{\textbf{v}}b_1.f,\dots,S_{\textbf{v}}b_r.f) = \sigma_\textbf{v}((b_1.f,\dots,b_r.f)).
\end{equation}
\end{definition}
\begin{theorem}
\label{thm:Dfin_CMF}
Given a D-finite function $f(\textbf{x},\textbf{z})\in F$ with $\dim(\A'.f)=r$, and a basis $B$ of the image, the basis change matrices $\clM^f_\textbf{v}$ constitute a CMF of dimension $d$ and rank $r$ over the field $K(\textbf{z})$.
\end{theorem}
\begin{proof}
We show that the matrices $\clM^f_\textbf{v}$ satisfy the cocycle condition \eqref{eq:cocycle}. Since $B$ is a basis, it is enough to verify that
$$
\sigma_{\textbf{v}+\textbf{w}}(B)
= B\cdot \clM^f_{\textbf{v}+\textbf{w}}
= B\cdot \clM^f_{\textbf{v}}\cdot \sigma_\textbf{v}(\clM^f_{\textbf{w}}).
$$
Indeed,
$$
B\cdot \clM^f_{\textbf{v}}\cdot \sigma_\textbf{v}(\clM^f_{\textbf{w}})
= \sigma_\textbf{v}(B)\cdot \sigma_\textbf{v}(\clM^f_{\textbf{w}})
= \sigma_\textbf{v}(B\cdot \clM^f_{\textbf{w}})
= \sigma_{\textbf{v}+\textbf{w}}(B).
$$
\end{proof}
This construction is fundamental to understanding CMFs and shows that they are natural mathematical objects. In particular, coboundary equivalence corresponds to a change of basis in $\A'.f$.
\begin{prop}
    A CMF $\clM$ is coboundary equivalent to a CMF generated by a D-finite function $f$ (that is, $\clM\sim\clM^f$) if and only if $\clM$ can also be generated by $f$.
\end{prop}
\begin{proof}
First, different choices of basis of $\A'.f$ produce coboundary-equivalent CMFs. Let $B$ and $B'$ be two bases of $\A'.f$, and let the corresponding CMFs be $\clM^f$ and ${\clM^f}'$. If $A$ is the basis change matrix, so that
    $$(b_1.f,\dots,b_r.f)\cdot A = (b_1'.f,\dots,b_r'.f),$$
then
    $$(b_1.f,\dots,b_r.f)\cdot A \cdot {\clM_\textbf{v}^{f}}{}' = (b_1'.f,\dots,b_r'.f)\cdot {\clM_\textbf{v}^{f}}{}' = \sigma_\textbf{v}(B')$$
and therefore
    $$\sigma_\textbf{v}(B') = \sigma_\textbf{v}(BA) = \sigma_\textbf{v}(B)\cdot \sigma_\textbf{v}(A) = (b_1.f,\dots,b_r.f)\cdot \clM_\textbf{v}^f\cdot \sigma_\textbf{v}(A).$$
Hence
    $$A\cdot {\clM_\textbf{v}^{f}}{}' = \clM_\textbf{v}^f\cdot \sigma_\textbf{v}(A),$$
which shows that $\clM^f\sim {\clM^f}'$.

Conversely, if $\clM\sim \clM^f$ with coboundary matrix $A$, then defining $B'\coloneqq B\cdot A$ yields a basis that generates $\clM$.
\end{proof}
CMFs can also be viewed as gauge transformations.
\begin{definition}
    For D-finite $f(\textbf{x},\textbf{z})$ and a basis $B$ of $\A'.f$, define $\clM_{\theta_{z_i}}^f\in M_{r\times r}(K(\textbf{x},\textbf{z}))$ by
$$(b_1.f,\dots,b_r.f)\clM_{\theta_{z_i}}^f = (\theta_{z_i}b_1.f,\dots,\theta_{z_i}b_r.f)$$
\end{definition}
\begin{lemma}\label{lem:CMF_Gauge}
For all $\textbf{v}\in\Z^d$, $\clM_\textbf{v}^f$ is a gauge transformation matrix from the system $\theta_{z_i}.g = g\cdot\clM_{\theta_{z_i}}^f$ to the system $\theta_{z_i}.g = g\cdot \sigma_\textbf{v}(\clM_{\theta_{z_i}}^f)$.
\end{lemma}
\begin{proof}
We must show that $\clM_\textbf{v}^f\sigma_\textbf{v}(\clM^f_{\theta_{z_i}}) = \theta_{z_i}.\clM^f_\textbf{v}+\clM^f_{\theta_{z_i}}\clM^f_\textbf{v}$. That can be verified by checking how both sides act on $B$:
$$(b_1.f,\dots,b_r.f)\cdot\clM_\textbf{v}^f\sigma_\textbf{v}(\clM^f_{\theta_{z_i}}) = \sigma_\textbf{v}((b_1.f,\dots,b_r.f)\cdot\clM^f_{\theta_{z_i}} )= (S_\textbf{v}\theta_{z_i}b_1.f,\dots,S_\textbf{v}\theta_{z_i}b_r.f)$$
$$(b_1.f,\dots,b_r.f)\cdot (\theta_{z_i}.\clM^f_\textbf{v}+\clM^f_{\theta_{z_i}}\clM^f_\textbf{v}) = (b_1.f,\dots,b_r.f)\cdot \theta_{z_i}.\clM^f_\textbf{v}+(\theta_{z_i}b_1.f,\dots,\theta_{z_i}b_r.f)\cdot\clM^f_\textbf{v}=$$
$$=\theta_{z_i}.\left((b_1.f,\dots,b_r.f)\cdot \clM^f_\textbf{v}\right) = (S_\textbf{v}\theta_{z_i}b_1.f,\dots,S_\textbf{v}\theta_{z_i}b_r.f)$$
This gives the desired equality.
\end{proof}
\begin{examp}\label{ex:tricomi}
Consider Tricomi's function $U(x_1,x_2,z)$, which is often used in physics \cite{Mathews2022}. It is D-finite, and in $\A = \Q(x_1,x_2,z)[S_{x_1},S_{x_2}]$, its annihilator is:
$$\text{ann} (U) = \langle z S_{x_2}+\left(x_1^2-x_2 x_1+x_1\right) S_{x_1}+\left(-x_1-z\right),z S_{x_2}^2+S_{x_2} \left(-x_2-z\right)+\left(x_2-x_1\right)\rangle$$

Thus, if we select $B = (U,S_{x_2}.U)$, we can calculate the generators of $\clM^U$:
$$S_{x_2}.B = (S_{x_2}.U,S_{x_2}^2.U)= (U,S_{x_2}.U)\begin{pmatrix}
    0&\frac{x_1-x_2}{z}\\
    1&\frac{x_2+z}{z}
\end{pmatrix}  = B\cdot \clM^U_{(0,1)}$$
$$S_{x_1}.B = (S_{x_1}.U,S_{x_1}S_{x_2}.U)= (U,S_{x_2}.U)\begin{pmatrix}
    \frac{x_1+z}{x_1^2-x_2 x_1+x_1}&-\frac{1}{x_1}\\
    \frac{-z}{x_1^2-x_2 x_1+x_1}&\frac{1}{x_1}
\end{pmatrix}  = B\cdot \clM^U_{(1,0)}$$
Hence, $\clM^U$ is a conservative matrix field of dimension $2$ and rank $2$ over $\Q(z)$.\\
Furthermore, $U$ satisfies:
$$\theta_z.U(x_1,x_2,z) = z(U(x_1,x_2,z)-U(x_1,x_2+1,z))$$
Implying that for the extended algebra $\A = \Q(x_1,x_2,z)[S_{x_1},S_{x_2},\theta_z]$ we have:
$$\theta_z\equiv z-zS_{x_2} \mod \text{ann}(U)\quad\text{or}\quad\theta_z.U = (z-zS_{x_2}).U$$
Hence, from linearity, the matrix $\clM^U_{\theta_z}$ is:
$$\clM^U_{\theta_z} = zI-z\clM^U_{(0,1)} = \begin{pmatrix}
    z&x_2-x_1\\
    -z&-x_2
\end{pmatrix}$$
\end{examp}
\begin{examp}
Consider the binomial function $f(x_1,x_2)=\binom{x_1}{x_2}$. This is a D-finite function, with annihilator
$$\text{ann}(f) = \langle(1+x_1-x_2)S_{x_1}-(1+x_1),(1+x_2)S_{x_2}-(x_1-x_2)\rangle$$
Since $\dim \A'.f = 1$, taking $B = (f)$ gives $\clM^f_{(1,0)} = 1+\frac{x_2}{1+x_1-x_2}$ and $\clM^f_{(0,1)} = \frac{x_1-x_2}{1+x_2}$. Thus $\clM^f$ is a conservative matrix field of dimension $2$ and rank $1$ over $\Q$ (that is, a multivariate hypergeometric term; see Remark~\ref{rem:MultiHyperTerm}).
\end{examp}
A particularly useful CMF of this kind is generated by the generalized hypergeometric function $_pF_q$.\\
In the case $z\notin\set{0,1}$, the annihilator of $\pFq{p}{q}{x_1,\dots,x_p}{x_{p+1},\dots,x_{p+q}}{z}$ is generated by its ODE, and its contiguous relations:
\begin{equation}
    \label{eq:cont_rel}
    \text{ann}\left(\pFq{p}{q}{x_1,\dots,x_p}{x_{p+1},\dots,x_{p+q}}{z}\right) = \left\langle\sum_{k=0}^{r}t_k(\textbf{x},z)\theta_z^k,\frac{\theta_z}{x_i}+1-S_{e_i},\frac{\theta_z}{x_{p+j}-1}+1-S_{-e_{p+j}}\right\rangle
\end{equation}
with $r = \max(p,q+1)$ the order of the differential equation, and $t_k\in\Q(\textbf{x},z)$ with $t_r = 1$. If we choose the basis $B = ({}_pF_q,\theta_z.{}_pF_q,\theta_z^2.{}_pF_q,\dots,\theta_z^{r-1}.{}_pF_q)$, then $\clM^{{}_pF_q}_{\theta_z}$ is in companion form:
$$\clM^{{}_pF_q}_{\theta_z} = \begin{pmatrix}
0 & 0 & \cdots & 0 & -t_0(\mathbf{x}, z) \\
1 & 0 & \cdots & 0 & -t_1(\mathbf{x}, z) \\
0 & 1 & \cdots & 0 & -t_2(\mathbf{x}, z) \\
\vdots & \vdots & \ddots & \vdots & \vdots \\
0 & 0 & \cdots & 1 & -t_{r-1}(\mathbf{x}, z)
\end{pmatrix}$$
and directly from the contiguous relations in \eqref{eq:cont_rel}, one gets
\begin{equation}
    \label{eq:pfq_cmf}
    \forall_{1\<=i\<=p}\clM^{{}_pF_q}_{e_i} =\frac{1}{x_i}\clM^{{}_pF_q}_{\theta_z}+I\ , \ \forall_{p+1\<=j\<=p+q}\clM^{{}_pF_q}_{-e_j}=\frac{1}{x_j-1}\clM^{{}_pF_q}_{\theta_z}+I.
\end{equation}
\begin{examp}\label{ex:2f1_generation}
Let us construct the CMF for $\pFq{2}{1}{x_1,x_2}{x_3}{z}$. First, note $r = 2$. The differential equation is:
$$(1-z)\theta_z^2 .\pFq{2}{1}{x_1,x_2}{x_3}{z}= x_1x_2z\pFq{2}{1}{x_1,x_2}{x_3}{z}+((x_1+x_2)z+1-x_3)\theta_z.\pFq{2}{1}{x_1,x_2}{x_3}{z}$$
Meaning, setting $B = ({}_2F_1,\theta_z.{}_2F_1)$, we get:
$$\clM^{{}_2F_1}_{\theta_z} = \begin{pmatrix}
 0 & \frac{x_1 x_2 z}{1-z} \\
 1 & \frac{x_1 z+x_2 z-x_3+1}{1-z} \\
\end{pmatrix}$$
Giving the generators of $\clM^{{}_2F_1}$ using \eqref{eq:pfq_cmf}:
$$\clM^{{}_2F_1}_{e_1} = 
\begin{pmatrix}
 1 & \frac{x_2 z}{1-z} \\
 \frac{1}{x_1} & \frac{x_1+x_2 z-x_3+1}{x_1(1-z)} \\
\end{pmatrix}\quad \clM^{{}_2F_1}_{e_2} = \begin{pmatrix}
 1 & \frac{x_1 z}{1-z} \\
 \frac{1}{x_2} & \frac{x_1 z+x_2-x_3+1}{x_2(1-z)} \\
\end{pmatrix} $$$$\clM^{{}_2F_1}_{e_3} = \sigma_3(\clM^{{}_2F_1}_{-e_3})^{-1} = \sigma_3\left(\frac{1}{x_3-1}\clM^{{}_2F_1}_{\theta_z}+I \right)^{-1}= \frac{x_3}{(x_1-x_3)(x_2-x_3)}\begin{pmatrix}
 -x_1-x_2+x_3 & x_1 x_2 \\
 \frac{1}{z}-1 & \frac{x_3 (z-1)}{z} \\
\end{pmatrix}. $$
Note that this is the CMF from Example~\ref{ex:2f1CMF_first_example}.
\end{examp}
Finally, let us discuss the matrix coefficients of CMFs.
\begin{lemma}\label{lem:MatrixCoeffsDFinite}
    Let $\clM$ be a CMF of dimension $d$ and rank $r$ over $\Q$, and let $\textbf{p},\textbf{p}'\in \Q^r$. Consider
\[
f(\textbf{v}):\Z^d\to \Q(\textbf{x})\qquad f(\textbf{v})\coloneqq \textbf{p}^{\mathsf T}\clM_\textbf{v}(\textbf{x})\textbf{p}'
\]
Then $f(\textbf{v})$ is a D-finite function in $\A = \Q(\textbf{v},\textbf{x})[S_{v_1},\dots,S_{v_d}]$. In fact, $\dim_{\Q(\textbf{v},\textbf{x})}\A.f\le r$.
\end{lemma}
\begin{proof}
From the closure properties of D-finite functions, only the basic cases $\textbf{p} = e_i,\textbf{p}' = e_j$ need to be resolved. Define $f(\textbf{v}) =e_i^{\mathsf T}\clM_{\textbf{v}}(\textbf{x})e_j $, and consider the row vector $e_i^{\mathsf T}\clM_{\textbf{v}}(\textbf{x})$. Note that from the co-cycle equation \eqref{eq:cocycle}:
\[
e_i^{\mathsf T}\clM_{\textbf{v}+\textbf{w}}(\textbf{x}) = e_i^{\mathsf T}\clM_{\textbf{v}}(\textbf{x})\clM_{\textbf{w}}(\textbf{x}+\textbf{v}).
\]
Since, for any $\textbf{w}\in \Z^d$, the entries of $\clM_{\textbf{w}}(\textbf{x}+\textbf{v})$ are rational functions of $\textbf{v}$, it follows that $S_\textbf{w}f(\textbf{v})$ is a $\Q(\textbf{v},\textbf{x})$-linear combination of the entries of the row vector $e_i^{\mathsf T}\clM_{\textbf{v}}(\textbf{x})$. Hence
\[
\dim_{\Q(\textbf{v},\textbf{x})}\A.f\le r.
\]
\end{proof}
Note that for CMFs generated by a D-finite function $\clM^f$, the matrix coefficients belong to $\A.g$ for a suitable solution $g$ of the difference equations satisfied by $f$.
\begin{examp}\label{ex:BeukersZeta2}
    Consider the Beukers integral \cite{Beukers1979} generating linear forms in $\zeta(2)$:
    \[
    I(\textbf{v})
=
\int_0^1 \int_0^1
\frac{x^{v_1}(1-x)^{v_2} y^{v_4}(1-y)^{v_3}}{(1-xy)^{\,v_5-v_2-v_3+1}}
\,dx\,dy
    \]
It is known that on a suitable positivity cone, this integral is D-finite with respect to shifts in its parameters. Hence one may construct $\clM^I$; using the basis $B = (I,S_{1,1,1,1,1} I)$, we obtain
\begin{equation}
\label{eq:Integral from CMF}
    I(\textbf{v}) = (I(\textbf{0}),I(\textbf{1}))\cdot \clM_\textbf{v}^I(0)\cdot e_1
 = (\zeta(2),5-3\zeta(2))\cdot \clM_\textbf{v}^I(0)\cdot e_1
\end{equation}
To get the coefficient of $1$ or of $\zeta(2)$ in these linear forms, simply use the appropriate row vector
\[
I(\textbf{v}) = (0,5)\cdot \clM_\textbf{v}^I(0)\cdot e_1
 +(1,-3)\cdot \clM_\textbf{v}^I(0)\cdot e_1 \zeta(2)\]
\end{examp}
The identity in \eqref{eq:Integral from CMF} can be generalized to give an efficient way to compute D-finite functions. Moreover, the CMF structure itself allows one to compute asymptotic expansions of matrix coefficients, and consequently of D-finite functions. This is developed in the following section.
\clearpage

\section{Asymptotics of CMFs}\label{sec:CMF_Asymp}
In this section, we study a specific subclass of conservative matrix fields and derive methods for computing the asymptotic growth of their matrix coefficients. Namely, in Corollary \ref{cor:cmfAsymptoticFactorization} we prove that every conservative matrix field $\clM$ in a certain class admits, along suitable directions $\textbf{v}\in\Z^d$, an asymptotic factorization of the form:
\[
     \clM_{n\textbf{v}}(\textbf{x}) = B(\textbf{x},\textbf{v})\diag(\lambda_i(\textbf{v})^n n^{\gamma_i(\textbf{x},\textbf{v})})\,(I+o(1))\,A(\textbf{v}) \qquad (n\to \infty)
,\]
for some explicit $B,\lambda_i,\gamma_i,A$. In Proposition \ref{prop:continuousEigenvalues} we demonstrate when $\frac{\log\abs{\lambda_1(\textbf{v})}}{\abs{\textbf{v}}}$ are continuous functions of $\frac{\textbf{v}}{\abs{\textbf{v}}}$. This is culminated in \eqref{eq:MatrixCoeffAsymp}, to get an asymptotic expansion of the matrix coefficients of the CMF. 

Much work has been done in the study of the asymptotics of solutions to difference and differential systems: \cite{Poincare1885,Perron1921,WimpZeilberger1985,BenzaidLutz1987}, including in the multivariate case \cite{Aomoto1974}. We introduce here a similar form of asymptotic analysis for a specific class of CMFs:
\begin{definition}
Let $\clM$ be a CMF of dimension $d$ and rank $r$ over $\Q$. We call $\clM$ \textit{balanced} if the entries of each generator $M_i$ of $\clM$, as well as those of its inverse $M_i^{-1}$, have nonpositive degree. Equivalently, the rational functions $h_{\pm i,j,k}(\textbf{x}) \coloneqq e_j^{\mathsf T}\clM_{\pm e_i}e_k\in \Q(\textbf{x})$ have total degree in the numerator less than or equal to that in the denominator.
\end{definition}
\begin{examp}
    The CMFs in Examples~\ref{ex:Zeta3CMF} and \ref{ex:Const3x3} are balanced. However, the CMF in Example~\ref{ex:2f1CMF_first_example} is not. Note that this property is not preserved under coboundary equivalence, as seen in Example~\ref{ex:Zeta3CoboundUnbalanced}.
\end{examp}
One useful consequence of a CMF being balanced is that, for all directions $\textbf{v}$ in some open subset of $\P^{d-1}$, the matrices $\clM_{e_i}(\textbf{x}+n\textbf{v})$ converge entrywise to matrices $N_i(\textbf{v})\in\GL_r(\R)$ as $n\to\infty$, with error term $O(1/n)$. When $\textbf{v}\in\Z^d$, the same is therefore true for the trajectory matrix $T_{\textbf{x},\textbf{v}}(n)$. This lets us apply the Benzaid--Lutz framework, i.e. a discrete Levinson-type method, to obtain strong asymptotic control of $\clM_{n\textbf{v}}(\textbf{x})$.
\begin{prop}\label{prop:generatorAsymp}
    Let $\clM$ be a balanced CMF of dimension $d$ and rank $r$ over $\Q$. Then, for all directions $\textbf{v}$ in some open subset of $\P^{d-1}$, the matrices $\clM_{e_i}(\textbf{x}+n\textbf{v})$ satisfy entrywise
    \[
    \clM_{e_i}(\textbf{x}+n\textbf{v}) = N_i(\textbf{v})+\frac{C}{n}+O(n^{-2})\qquad (n\to\infty)
    \]
    for some  $N_i(\textbf{v})\in\GL_r(\R)$. Moreover,
    \[
     N_i(\textbf{v}) N_j(\textbf{v}) = N_j(\textbf{v}) N_i(\textbf{v})\qquad \forall\, 1\le i,j\le d
    \]
\end{prop}
\begin{proof}
Looking at the entries of $\clM_{e_i}(\textbf{x}+n\textbf{v})$ and $\clM_{-e_i}(\textbf{x}+n\textbf{v})$, we see that as long as $\textbf{v}$ is not a root or singularity of the leading homogeneous terms of these rational functions, the entries become rational functions of $n$ of balanced degree. Hence they admit the stated asymptotic expansion. The set of such $\textbf{v}$ is open in $\P^{d-1}$.

The limit matrices $N_i(\textbf{v})$ are invertible because of the relation between $\clM_{e_i}^{-1}$ and $\clM_{-e_i}$. Finally, \eqref{eq:cocycle} gives
\[
\clM_{e_i}(\textbf{x}+n\textbf{v})\clM_{e_j}(\textbf{x}+n\textbf{v}+e_i) = \clM_{e_j}(\textbf{x}+n\textbf{v})\clM_{e_i}(\textbf{x}+n\textbf{v}+e_j),
\]
and taking the limit as $n\to\infty$ yields
 \[
     N_i(\textbf{v}) N_j(\textbf{v}) = N_j(\textbf{v}) N_i(\textbf{v})\qquad \forall\, 1\le i,j\le d.
    \]
\end{proof}
\begin{corollary}
\label{cor:TrajAsymp}
For $\clM$ and $\textbf{v} = (v_1,\dots,v_d)\in\Z^d$ as in Proposition~\ref{prop:generatorAsymp}, the trajectory matrix $T_{\textbf{x},\textbf{v}}(n)$ has the following entrywise asymptotic expansion:
\[
T_{\textbf{x},\textbf{v}} (n) = \prod_{i=1}^d N_i(\textbf{v})^{v_i}+\frac{C}{n}+O(n^{-2}) \qquad (n\to\infty)
\]
\end{corollary}
\begin{proof}
    This is immediate from the definition of the trajectory matrix and Proposition~\ref{prop:generatorAsymp}.
\end{proof}
Now that we have an asymptotic expansion of trajectory matrices, we can use a Levinson/Benzaid--Lutz type factorization to obtain an asymptotic factorization of $\clM_{n\textbf{v}}(\textbf{x})$, at least under suitable hypotheses on the eigenvalues.
\begin{lemma}\label{lem:LutzAsympFactorization}
    Suppose that a sequence of matrices $T(n)\in \operatorname{GL}_r(\Q(n))$ satisfies
    \[
    T(n) = T+\frac{C}{n}+O(n^{-2}) \qquad (n\to\infty)
    \]
    and that the eigenvalues of $T$ satisfy
    \[    \abs{\lambda_1}>\dots>\abs{\lambda_r}>0
    \]
    Then the product $P(n)\coloneqq T(1)T(2)\cdots T(n)$ factors as
    \[
    B\operatorname{diag}(\lambda_1^n n^{\gamma_1},\dots,\lambda_r^n n^{\gamma_r})(I+o(1))A
    \]
for some $B,A\in\GL_r(\R),\gamma_i\in \R$.

Moreover, if $S^{-1}TS=\Lambda\coloneqq\diag(\lambda_1,\dots,\lambda_r)$ and $\widetilde C\coloneqq S^{-1}CS$,
then one may take
\[
\gamma_i=\frac{d_i}{\lambda_i},\qquad d_i\coloneqq(\widetilde C)_{ii},\qquad A = S^{-1}.
\]
\end{lemma}

\begin{proof}
In this proof, we find a coboundary transformation $H(n)$ that diagonalizes $T(n)$ up to order $O(n^{-2})$, then calculate the cummulative product $\Phi(n)$ of the diagonal terms $A(n)$,  and use a result by Benzaid and Lutz \cite{BenzaidLutz1987} to connect $\Phi(n)$ and $P(n)$. The asymptotic factorization is then done on $\Phi(n)$, and propagated to the form above. 

Since $T$ has real entries, the strict inequalities $|\lambda_1|>\cdots>|\lambda_r|$ 
force $\lambda_i\in\R$. In particular, $T$ has simple real spectrum and is diagonalizable over $\R$.
Fix $S\in\GL_r(\R)$ such that $S^{-1}TS=\Lambda=\diag(\lambda_1,\dots,\lambda_r)$.

First - we conjugate the system:

Set $\widetilde T(n)\coloneqq S^{-1}T(n)S$. Conjugating the expansion gives
\[
\widetilde T(n)=\Lambda+\frac{\widetilde C}{n}+O(n^{-2}),\qquad \widetilde C\coloneqq S^{-1}CS.
\]
Write $\widetilde C=D+N$ with $D\coloneqq\diag(\widetilde C)$ diagonal and $N\coloneqq\widetilde C-D$ off-diagonal. Define a constant matrix $U$ by $U_{ii}\coloneqq0$ and, for $i\neq j$,
\[
U_{ij}\coloneqq-\frac{N_{ij}}{\lambda_i-\lambda_j}.
\]
Then $\Lambda U-U\Lambda=-N$. Define
\[
H(n)\coloneqq S\Bigl(I+\frac{U}{n}\Bigr)=S+\frac{SU}{n}+O(n^{-2}).
\]
A direct expansion yields
\begin{equation}\label{eq:hatT}
\widehat T(n)\coloneqq H(n)^{-1}T(n)H(n+1)=\Lambda+\frac{D}{n}+R(n).
\end{equation}
Where
\[
R(n)=O(n^{-2})\qquad (n\to\infty).
\]
From $T(n)=H(n)\widehat T(n)H(n+1)^{-1}$ we obtain
\begin{equation}\label{eq:telescope}
P(n)=H(1)\,\widehat P(n)\,H(n+1)^{-1},
\qquad
\widehat P(n)\coloneqq\widehat T(1)\widehat T(2)\cdots \widehat T(n).
\end{equation}
This can also be thought of, in our context, as a coboundary transformation (see definition \ref{def:coboundary}).

Let
\[
A(n)\coloneqq\Lambda+\frac{D}{n}\quad(\text{diagonal}),\qquad
\Phi(n)\coloneqq A(1)A(2)\cdots A(n)\quad(\text{diagonal}).
\]
Since $R(n)=O(n^{-2})$ is summable, we apply a discrete diagonal-Levinson theorem
in the \emph{column convention} to the transposed system
\[
x(n+1)=\widehat T(n+1)^{\mathsf T}x(n).
\]
Writing $Q(n)\coloneqq\widehat P(n)^{\mathsf T}$ for its fundamental matrix,
Lemma 2.1 in \cite{BenzaidLutz1987} yields a fundamental matrix $W(n)$ of the form
\[
W(n)=(I+o(1))\,\Phi(n),
\]
hence $Q(n)=W(n)\,C_\infty$ for some constant $C_\infty\in\GL_r(\R)$.
Transposing back, we obtain
\begin{equation}\label{eq:BLform}
\widehat P(n)=C_\infty^{\mathsf T}\,\Phi(n)\,(I+o(1)).
\end{equation}

Now, let us factorize $\Phi(n)$ asymptotically. Write $D=\diag(d_1,\dots,d_r)$. Since $\Phi(n)$ is diagonal,
\[
\Phi(n)=\diag\Bigl(\prod_{k=1}^n\bigl(\lambda_i+d_i/k\bigr)\Bigr)_{i=1}^r.
\]
For each $i$,
\[
\prod_{k=1}^n\Bigl(\lambda_i+\frac{d_i}{k}\Bigr)
=\lambda_i^n\prod_{k=1}^n\Bigl(1+\frac{d_i}{\lambda_i k}\Bigr)
=c_i\,\lambda_i^n\,n^{d_i/\lambda_i}\,(1+o(1))
\]
for some $c_i\neq 0$. Hence, with $\gamma_i\coloneqq d_i/\lambda_i$ and $C_0\coloneqq\diag(c_1,\dots,c_r)\in\GL_r(\R)$,
\begin{equation}\label{eq:Phi-asymp}
\Phi(n)=C_0\,\diag(\lambda_1^n n^{\gamma_1},\dots,\lambda_r^n n^{\gamma_r})\,(I+o(1)),
\end{equation}
where the right-hand $(I+o(1))$ is diagonal.
Combining \eqref{eq:telescope}, \eqref{eq:BLform}, and \eqref{eq:Phi-asymp} gives
\[
P(n)=H(1)\,C_\infty^{\mathsf T}\,C_0\,\diag(\lambda_i^n n^{\gamma_i})\,(I+o(1))\,H(n+1)^{-1}.
\]
Since $H(n+1)^{-1}=S^{-1}(I+O(1/n))$ tends to $S^{-1}$ and is bounded, we may absorb it
into the same right-hand $(I+o(1))$ factor, obtaining
\[
P(n)=B\,\diag(\lambda_i^n n^{\gamma_i})\,(I+o(1))\,A
\]
with $B\coloneqq H(1)\,C_\infty^{\mathsf T}\,C_0\in\GL_r(\R)$ and $A\coloneqq S^{-1}\in\GL_r(\R)$.
\end{proof}
Let us now apply this result to balanced CMFs. In our setting, we are naturally interested in the specialization $T(n)=T_{\textbf{x},\textbf{v}}(n)$, a trajectory matrix, so that by \eqref{eq:prod_of_traj} we have $P(n)=\clM_{n\textbf{v}}(\textbf{x})$.
\begin{corollary}
\label{cor:cmfAsymptoticFactorization}
    For $\clM$ and $\textbf{v}$ as in Proposition~\ref{prop:generatorAsymp}, assume that the eigenvalues $\lambda_i(\textbf{v})$ of $T\coloneqq \lim_{n\to \infty} T_{\textbf{x},\textbf{v}}(n)$ satisfy $\abs{\lambda_1(\textbf{v})}>\dots>\abs{\lambda_r(\textbf{v})}>0$. Then one obtains an asymptotic factorization of $\clM_{n\textbf{v}}(\textbf{x})$:
    \[
     \clM_{n\textbf{v}}(\textbf{x}) = B(\textbf{x},\textbf{v})\diag(\lambda_i(\textbf{v})^n n^{\gamma_i(\textbf{x},\textbf{v})})\,(I+o(1))\,A(\textbf{v})
    \]
\end{corollary}
Corollary~\ref{cor:cmfAsymptoticFactorization} allows one to compute the asymptotics of the matrix coefficients of the CMF. For $f(\textbf{v}) = \textbf{p}^{\mathsf T}\clM_{\textbf{v}}(\textbf{x})\textbf{p}'$, we get
\[
f(n\textbf{v}) = \textbf{p}^{\mathsf T}B(\textbf{x},\textbf{v})\diag(\lambda_i(\textbf{v})^n n^{\gamma_i(\textbf{x},\textbf{v})})\,(I+R(n))\,A(\textbf{v})\textbf{p}'
\qquad R(n)_{i,j} = o(1)\]
Denote $(\alpha_1,\dots,\alpha_r)\coloneqq A(\textbf{v})\textbf{p}'$, $(\beta_1,\dots,\beta_r)\coloneqq \textbf{p}^{\mathsf T}B(\textbf{x},\textbf{v})$, and $k\coloneqq \min\set{i:\beta_i\neq 0}$. Assuming $\forall_i \alpha_i \neq 0$, it is simple to show that 
\begin{equation}\label{eq:MatrixCoeffAsymp}
    f(n\textbf{v})= \alpha_k\beta_k\lambda_k(\textbf{v})^n n^{\gamma_k(\textbf{x},\textbf{v})}(1+E(n))+O(\lambda_{j}(\textbf{v})^n n^{\gamma_{j}(\textbf{x},\textbf{v})})
\end{equation}
where $E(n)=o(1)$ comes from $R(n)$, and $j = \min\set{i:\beta_i\neq 0,\ i>k}$ (if $k=r$, the final error term on the right is omitted).
The assumption $\alpha_i\neq 0$ is harmless in the context of this paper, since it concerns constant choices of $\textbf{p}'$, which can be selected to satisfy this condition.

Finally, one can learn more about the behavior of $\lambda_i(\textbf{v})$ as functions of $\textbf{v}$ by using Corollary ~\ref{cor:TrajAsymp}.
\begin{prop}
\label{prop:continuousEigenvalues}
    Let $\clM$ and $\textbf{v}$ be as in Corollary~\ref{cor:TrajAsymp}. Assume further that $N_i(\textbf{v}) \coloneqq \lim_{n\to\infty} \clM_{e_i}(\textbf{x}+n\textbf{v})$ is diagonalizable for all $i$. Then the normalized spectrum of $T = \lim_{n\to\infty} T_{\textbf{x},\textbf{v}}(n)$, defined by
    \[
    \left (\frac{\log\abs{\lambda_1(\textbf{v})}}{\abs{\textbf{v}}},\frac{\log\abs{\lambda_2(\textbf{v})}}{\abs{\textbf{v}}},\dots,\frac{\log\abs{\lambda_r(\textbf{v})}}{\abs{\textbf{v}}}\right )
    \]
    is a continuous function of $\frac{\textbf{v}}{\abs{\textbf{v}}}$.
\end{prop}
\begin{proof}
First, note that as $N_i(\textbf{v})$ is just the value of plugging $\textbf{v}$ in the leading monomials of the generator matrices of the CMF, their entries are continuous functions of $\frac{\textbf{v}}{\abs{\textbf{v}}}$.

As the matrices $N_i(\textbf{v})$ commute, there is a matrix $S(\textbf{v})$ that diagonalizes all of them simultaneously. Let us denote
\[
\Lambda_i \coloneqq S^{-1}(\textbf{v})N_i(\textbf{v}) S(\textbf{v})
\]
Thus, as $T =\prod_{i=1}^d N_i(\textbf{v})^{v_i} $, we have that:
    \[
    S^{-1}(\textbf{v})T S(\textbf{v}) = S^{-1}(\textbf{v})\prod_{i=1}^d N_i(\textbf{v})^{v_i}  S(\textbf{v}) = \prod_{i=1}^d \Lambda_i ^{v_i}
    \]

By direct calculation, each normalized eigenvalue of $T$ can be written as an inner product of $\frac{\textbf{v}}{\abs{\textbf{v}}}$ with the logarithms of the absolute values of the corresponding eigenvalues of $N_i(\textbf{v})$. This is continuous as a function of $\frac{\textbf{v}}{\abs{\textbf{v}}}$.
\end{proof}
The next section uses these tools—namely Corollary~\ref{cor:cmfAsymptoticFactorization}, Proposition~\ref{prop:continuousEigenvalues}, and \eqref{eq:MatrixCoeffAsymp}—to analyze the linear forms and approximations arising from CMFs.
\clearpage
\section{CMF Ratios and CMF Limits}\label{sec:CMF_Ratios}
This section defines and discusses CMF limits and CMF ratios, which generalize Ap\'ery limits and D-finite ratios. We establish their limits, and convergence rates (in the sense of Definition \ref{def:convergence_rate}) using the tools established in the previous section, namely the asymptotic expansion in \eqref{eq:MatrixCoeffAsymp}.
\begin{definition}
Let $\clM$ be a CMF of dimension $d$ and rank $r$ over $\Q$. Choose a trajectory $\textbf{x}+n\textbf{v}$ with $\textbf{v}\in\Z^d$ and $\textbf{x}\in\Q^d$, together with four vectors $\textbf{p},\textbf{p}',\textbf{q},\textbf{q}'\in \Q^r$.\\
The \textit{CMF ratio} is then the sequence
    \begin{equation} \label{eq:CMF_limit_def}
        \clL_{\textbf{x},\textbf{v}}^{\textbf{p},\textbf{p}',\textbf{q},\textbf{q}'}(n) \coloneqq \frac{\textbf{p}^{\mathsf T}\cdot \clM_{n\textbf{v}}(\textbf{x})\cdot \textbf{p}'}{\textbf{q}^{\mathsf T}\cdot \clM_{n\textbf{v}}(\textbf{x})\cdot \textbf{q}'}.
    \end{equation}
For brevity, write $\clL_{\textbf{x},\textbf{v}}^{\textbf{p},\textbf{q}}\coloneqq\clL_{\textbf{x},\textbf{v}}^{\textbf{p},e_r,\textbf{q},e_r}$. The limit of $\clL_{\textbf{x},\textbf{v}}^{\textbf{p},\textbf{p}',\textbf{q},\textbf{q}'}(n)$ as $n\to\infty$ is called a \textit{CMF limit}.
\end{definition}
Given Lemma~\ref{lem:MatrixCoeffsDFinite}, each specific CMF limit is also an Ap\'ery limit in the sense of Definition~\ref{def:AperyLimit}. However, the present terminology is useful when the direction $\textbf{v}$ varies. We begin with two concrete examples and then use the results of the previous section to describe their limits and convergence rates.
\begin{examp}
Consider the CMF generated by $\pFq{2}{1}{x_1,x_2}{x_3}{-1}$. This CMF is obtained from the CMF in Examples~\ref{ex:2f1CMF_first_example} and \ref{ex:2f1_generation} by substituting $z=-1$. Let us compute a few of its CMF ratios.\\
First, consider $\clL_{\textbf{x},\textbf{v}}^{\textbf{p},\textbf{q}}(n)$ for $\textbf{x}=\textbf{v}=(1,1,2)$, $\textbf{p}=(0,1)$, and $\textbf{q}=(-2,2)$. We begin by computing $\clM_{n\textbf{v}}(\textbf{x})$ for a few values of $n$:
\[\clM_0(\textbf{x}) = I, \clM_\textbf{v}(\textbf{x}) = \begin{pmatrix}
    -6&18\\-24&66
\end{pmatrix},\clM_{2\textbf{v}}(\textbf{x}) = \begin{pmatrix}
    -150&660\\-540&2370
\end{pmatrix}, \clM_{3\textbf{v}}(\textbf{x}) =\begin{pmatrix}
    - \frac{10220}{3} & \frac{62300}{3}\\- \frac{36680}{3} & \frac{223580}{3}
\end{pmatrix}. \]
Hence:
$$\clL_{\textbf{x},\textbf{v}}^{\textbf{p},\textbf{q}}(0) = \frac{1}{2},\clL_{\textbf{x},\textbf{v}}^{\textbf{p},\textbf{q}}(1) = \frac{11}{16},\clL_{\textbf{x},\textbf{v}}^{\textbf{p},\textbf{q}}(2) = \frac{79}{114},\clL_{\textbf{x},\textbf{v}}^{\textbf{p},\textbf{q}}(3) = \frac{1597}{2304}.$$
Or in decimal form:
$$\clL_{\textbf{x},\textbf{v}}^{\textbf{p},\textbf{q}}(n) = (0.5,0.6875,0.6930\dots,0.69314\dots,\dots)$$
The sequence appears to converge to $\log(2)\approx 0.693147\dots$. Moreover, it has a nontrivial convergence rate $\rho = -3.53\dots$ and appears to have a positive irrationality measure $\delta \approx 0.31$.

In contrast, consider the CMF ratio $\clL_{\textbf{x},\textbf{v}}^{\textbf{p},\textbf{q}}(n)$ for $\textbf{x}=\textbf{v}=(-1,-1,2)$, $\textbf{p}=(0,1)$, and $\textbf{q}=(-2,2)$.
The first few values are:
$$\clL_{\textbf{x},\textbf{v}}^{\textbf{p},\textbf{q}}(n) = (\frac{1}{2},\frac{19}{60}, \frac{1109}{5460},\frac{713}{13860},- \frac{327713}{540540},\dots)$$
This sequence does not appear to converge. The first $300$ values are plotted in Figure~\ref{fig:2f1NonConverging}.
\begin{figure}[htbp]
    \centering
    \includegraphics[width=0.75\linewidth]{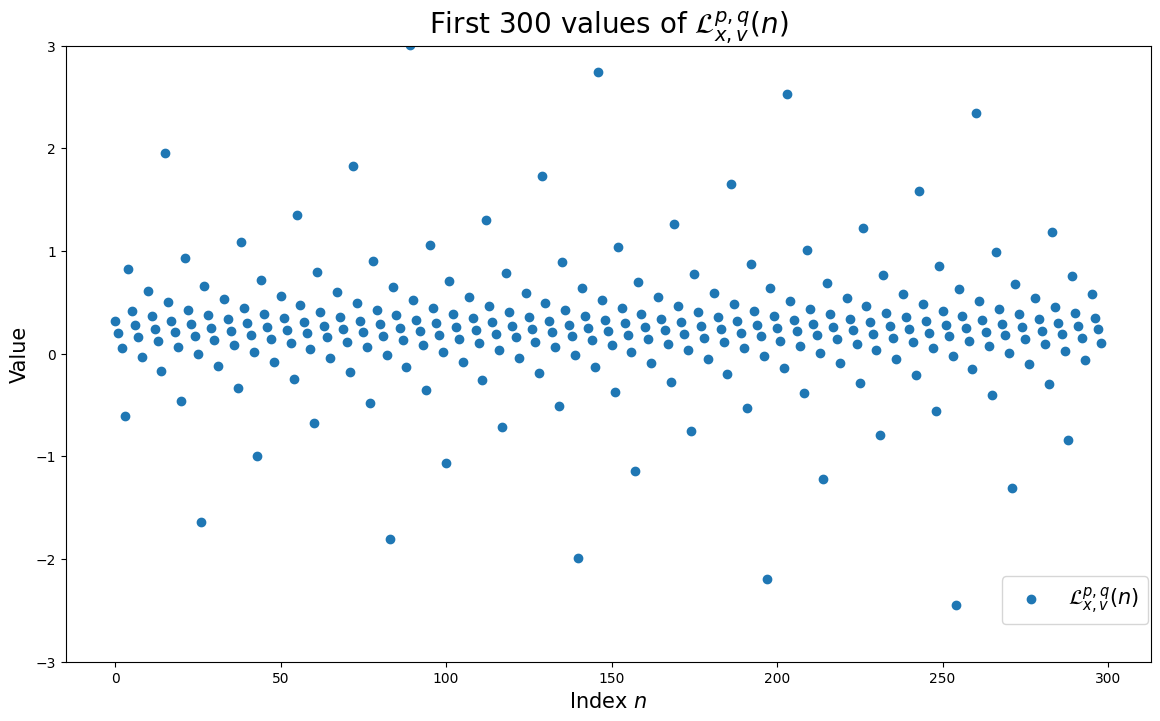}
    \caption{Example of a nonconvergent CMF ratio: the first $300$ values of $\clL_{\textbf{x},\textbf{v}}^{\textbf{p},\textbf{q}}(n)$.}
    \label{fig:2f1NonConverging}
\end{figure}
\end{examp}
From now on, let us consider CMFs $\clM$ and directions $\textbf{v}$ that satisfy the hypotheses of Proposition~\ref{prop:continuousEigenvalues}. In this setting, we may use \eqref{eq:MatrixCoeffAsymp}.
\begin{theorem}\label{thm:cmf_rat_convergence}
Let $\clM$ and $\textbf{v}$ be as in Proposition~\ref{prop:continuousEigenvalues}. Furthermore, assume that $e_r$ satisfies $\forall_i\ e_i^\mathsf{T}A(\textbf{v})e_r\neq 0$. Then the CMF limit satisfies
\[
\clL_{\textbf{x},\textbf{v}}^{\textbf{p},\textbf{q}}(n) = \frac{\textbf{p}^{\mathsf T}B(\textbf{x},\textbf{v})e_k}{\textbf{q}^{\mathsf T}B(\textbf{x},\textbf{v})e_k} + \Theta\left(\left(\frac{\lambda_j(\textbf{v})}{\lambda_k(\textbf{v})}\right)^n n^C\right)
\]
provided the following conditions hold:
\begin{enumerate}
    \item $\textbf{p}$ and $\textbf{q}$ are linearly independent (otherwise the error term may be omitted).
    \item $\min\set{i:\textbf{p}^{\mathsf T}B(\textbf{x},\textbf{v})e_i\neq 0} = k = \min\set{i:\textbf{q}^{\mathsf T}B(\textbf{x},\textbf{v})e_i\neq 0}$.
    \item $j = \min\set{i:(\textbf{p}-c\textbf{q})^{\mathsf T}B(\textbf{x},\textbf{v})e_i\neq 0}$, where $c = \frac{\textbf{p}^{\mathsf T}B(\textbf{x},\textbf{v})e_k}{\textbf{q}^{\mathsf T}B(\textbf{x},\textbf{v})e_k}$.
\end{enumerate}
\end{theorem}
\begin{proof}
Denote $c\coloneqq \frac{\textbf{p}^{\mathsf T}B(\textbf{x},\textbf{v})e_k}{\textbf{q}^{\mathsf T}B(\textbf{x},\textbf{v})e_k}$. A direct computation shows that multiplying the error $\clL_{\textbf{x},\textbf{v}}^{\textbf{p},\textbf{q}}(n)-c$ by the denominator of the CMF ratio produces a different matrix coefficient:
\[
\textbf{q}^{\mathsf{T}}\clM_{n \textbf{v}}(\textbf{x})e_r \left(\clL_{\textbf{x},\textbf{v}}^{\textbf{p},\textbf{q}}(n) -c \right) =(\textbf{p}-c\textbf{q})^{\mathsf T} \clM_{n \textbf{v}}(\textbf{x})e_r
\]
Condition 1 ensures that $\textbf{p}-c\textbf{q}\neq \textbf{0}$, and hence that $j$ from condition 3 exists. Condition 2 implies that $j>k$. Finally, using \eqref{eq:MatrixCoeffAsymp}, we obtain
\[
\textbf{q}^{\mathsf{T}}\clM_{n \textbf{v}}(\textbf{x})e_r =  \Theta(\lambda_k(\textbf{v})^n n^{C_1})  ,
\]
and 
\[
(\textbf{p}-c\textbf{q})^{\mathsf T} \clM_{n \textbf{v}}(\textbf{x})e_r = \Theta(\lambda_j(\textbf{v})^n n^{C_2})
\]
and therefore
\[
\clL_{\textbf{x},\textbf{v}}^{\textbf{p},\textbf{q}}(n) -c  =\frac{(\textbf{p}-c\textbf{q})^{\mathsf T} \clM_{n \textbf{v}}(\textbf{x})e_r}{\textbf{q}^{\mathsf{T}}\clM_{n \textbf{v}}(\textbf{x})e_r} = \Theta\left(\left(\frac{\lambda_j(\textbf{v})}{\lambda_k(\textbf{v})}\right)^n n^C\right)
\]
for some $C$.
\end{proof}
This is equivalent to the statement that, under the hypotheses of Theorem~\ref{thm:cmf_rat_convergence}, the CMF ratio $\clL_{\textbf{x},\textbf{v}}^{\textbf{p},\textbf{q}}(n)$ converges to the CMF limit
\[
c = \frac{\textbf{p}^{\mathsf T}B(\textbf{x},\textbf{v})e_k}{\textbf{q}^{\mathsf T}B(\textbf{x},\textbf{v})e_k}
\]
and has convergence rate $\rho = \log\abs{\lambda_j(\textbf{v})}-\log\abs{\lambda_k(\textbf{v})}$ in the sense of Definition~\ref{def:convergence_rate}.

The following section studies a number of experimental properties of specific CMFs and CMF limits, formulates them as conjectures, proves some special cases, and suggests several applications and corollaries.
\clearpage
\section{Experimental Analysis of Conservative Matrix Fields}\label{sec:Experiments}
In Section~\ref{sec:Preliminaries}, we introduced four quantities attached to sequences of Diophantine approximations: the limit, the convergence rate ($\rho$), the height growth ($\eta$), and the irrationality measure ($\delta$); see Definitions~\ref{def:convergence_rate}, \ref{def:limit_height}, and \ref{def:irrationality_measure_of_seq}. In this section, we analyze how these quantities, for CMF ratios $\clL_{\textbf{x},\textbf{v}}^{\textbf{p},\textbf{q}}(n)$, depend on the direction $\textbf{v}$. We do so both empirically and theoretically.

\begin{prop}\label{prop:homog_and_contin_metrics}
Consider a CMF $\clM$ and an open set $V$ of directions such that every $\textbf{v}\in V$ satisfies the hypotheses of Theorem~\ref{thm:cmf_rat_convergence}. For the family of CMF ratios $\clL_{\textbf{x},\textbf{v}}^{\textbf{p},\textbf{q}}(n)$, $\textbf{v}\in V$, denote, and assume the existence of, the limit $l(\textbf{v})$, the convergence rate $\rho(\textbf{v})$, the height growth $\eta(\textbf{v})$, and the irrationality measure $\delta(\textbf{v})$.

Then these quantities are homogeneous:
\[
l(k\textbf{v}) = l(\textbf{v})\quad \delta(k\textbf{v}) = \delta(\textbf{v})\quad \rho(k\textbf{v}) = k\rho(\textbf{v})\quad \eta(k\textbf{v}) = k\eta(\textbf{v})
\]
for all $k\in \N_{>0}$. Moreover, in the case $r = 2$ (with $r$ the rank of $\clM$), $\frac{\rho(\textbf{v})}{\abs{\textbf{v}}}$ is a continuous function of $\frac{\textbf{v}}{\abs{\textbf{v}}}$.
\end{prop}
\begin{proof}
    The homogeneity is immediate from $\clL_{\textbf{x},k\textbf{v}}^{\textbf{p},\textbf{q}}(n)=\clL_{\textbf{x},\textbf{v}}^{\textbf{p},\textbf{q}}(kn)$ together with Remark~\ref{rem:convergence_height_delta_connection}. The continuity of the normalized convergence rate follows from Proposition~\ref{prop:continuousEigenvalues} and the fact that, when $r=2$, the convergence rate of $\clL_{\textbf{x},\textbf{v}}^{\textbf{p},\textbf{q}}(n)$ must equal $\rho(\textbf{v}) = \log\abs{\lambda_2(\textbf{v})}-\log\abs{\lambda_1(\textbf{v})}$ by the linear independence of $\textbf{p}$ and $\textbf{q}$. For $r>2$, the same conclusion holds precisely when the indices $1\le k<j\le r$ in $\rho(\textbf{v}) = \log\abs{\lambda_j(\textbf{v})}-\log\abs{\lambda_k(\textbf{v})}$ do not depend on $\textbf{v}$.
\end{proof}

We now choose specific CMFs and parameters $\textbf{x},\textbf{p},\textbf{q}$ and graph empirical evaluations of these quantities. We sample $\textbf{v}$ from sets of lattice points in the unit ball for which $\forall_{n\in\N}\ \clM_{n\textbf{v}}(\textbf{x})\in\text{GL}_r(\Q)$. While $\frac{\log\abs{\lambda_i(\textbf{v})}}{\abs{\textbf{v}}}$ can be computed in closed form, the other parameters must be estimated numerically. For $N\gg 1$, we use
\[
\hat{l} = \clL^{\textbf{p},\textbf{q}}_{\textbf{x},\textbf{v}}(N)\approx\lim_{n\to\infty}\clL^{\textbf{p},\textbf{q}}_{\textbf{x},\textbf{v}}(n)=l,
\qquad
\hat{\rho} = \frac{\log\abs{\clL^{\textbf{p},\textbf{q}}_{\textbf{x},\textbf{v}}(N)-\clL^{\textbf{p},\textbf{q}}_{\textbf{x},\textbf{v}}(2N)}}{N}\approx\rho,
\]
and
\[
\hat{\delta} = -1-\frac{\log\abs{\clL^{\textbf{p},\textbf{q}}_{\textbf{x},\textbf{v}}(N)-\clL^{\textbf{p},\textbf{q}}_{\textbf{x},\textbf{v}}(2N)}}{\log\abs{H(\clL^{\textbf{p},\textbf{q}}_{\textbf{x},\textbf{v}}(N))}}\approx\delta.
\]
First, we examine the CMF from Example~\ref{ex:Zeta3CMF}, generated by
$$M_{1} =  \begin{pmatrix}0 & -1\\\frac{\left(x_{1} + 1\right)^{3}}{x_{1}^{3}} & \frac{x_{1}^{3} + 2 x_{2} \left(2 x_{1} + 1\right) \left(x_{2} - 1\right) + \left(x_{1} + 1\right)^{3}}{x_{1}^{3}}\end{pmatrix} \quad M_{2}= \begin{pmatrix}\frac{- x_{1}^{3} + 2 x_{1}^{2} x_{2} - 2 x_{1} x_{2}^{2} + x_{2}^{3}}{x_{2}^{3}} & - \frac{x_{1}^{3}}{x_{2}^{3}}\\\frac{x_{1}^{3}}{x_{2}^{3}} & \frac{x_{1}^{3} + 2 x_{1}^{2} x_{2} + 2 x_{1} x_{2}^{2} + x_{2}^{3}}{x_{2}^{3}}\end{pmatrix}$$
This CMF satisfies the hypotheses of Theorem~\ref{thm:cmf_rat_convergence} for all $\textbf{v}=(a,b)$ with $a,b>0$. We examine the ratio $\clL_{\textbf{x},\textbf{v}}^{\textbf{p},\textbf{q}}(n)$ for $\textbf{x}=(1,1)$, $\textbf{p}=(0,1)$, and $\textbf{q}=(1,1)$, with $\textbf{v}\in \set{(a,b)\in\N^2 : \gcd(a,b)=1,\ \abs{(a,b)}<14}$. For each of the $97$ trajectories, we computed $\hat{l}(\textbf{v})$, $\hat{\delta}(\textbf{v})$, $\frac{\hat{\rho}(\textbf{v})}{\abs{\textbf{v}}}$, and $\frac{\log\abs{\lambda_i(\textbf{v})}}{\abs{\textbf{v}}}$ with $N=1000$. These quantities are plotted as functions of the angle of $\textbf{v}$ with the $x_1$-axis in Figure~\ref{fig:zeta3panels}.\\
\begin{figure}[htbp]
    \centering

    \begin{subfigure}[b]{0.49\textwidth}
        \centering
        \includegraphics[width=\textwidth]{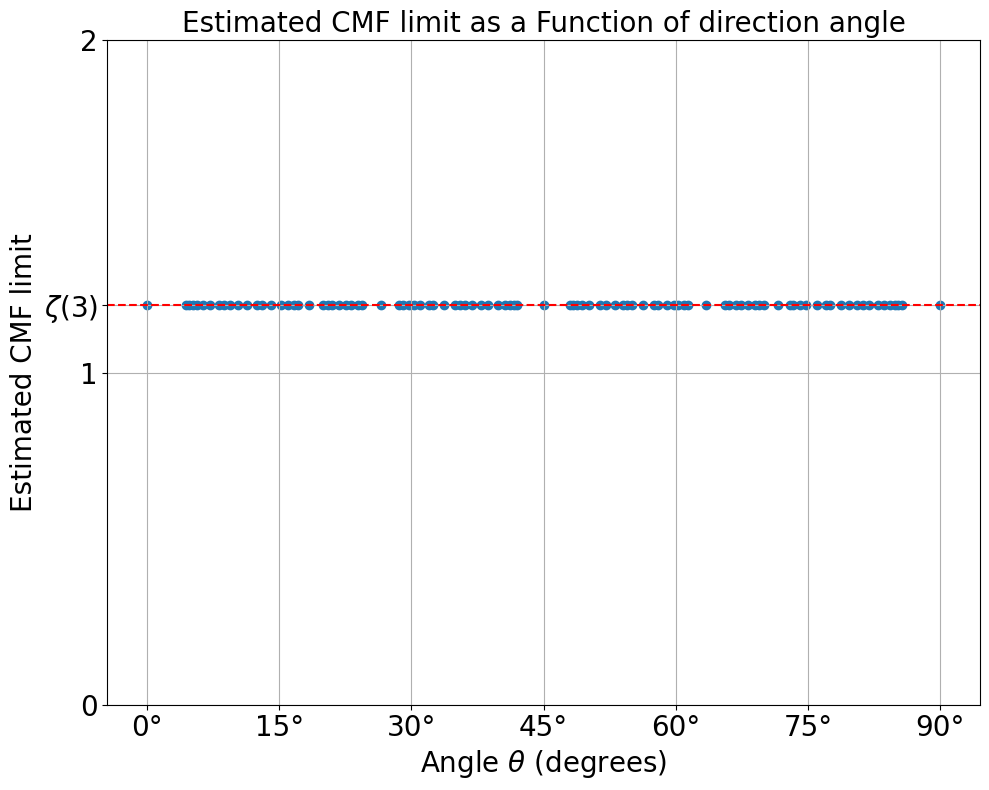}
    \end{subfigure}
    \hfill
    \begin{subfigure}[b]{0.49\textwidth}
        \centering
        \includegraphics[width=\textwidth]{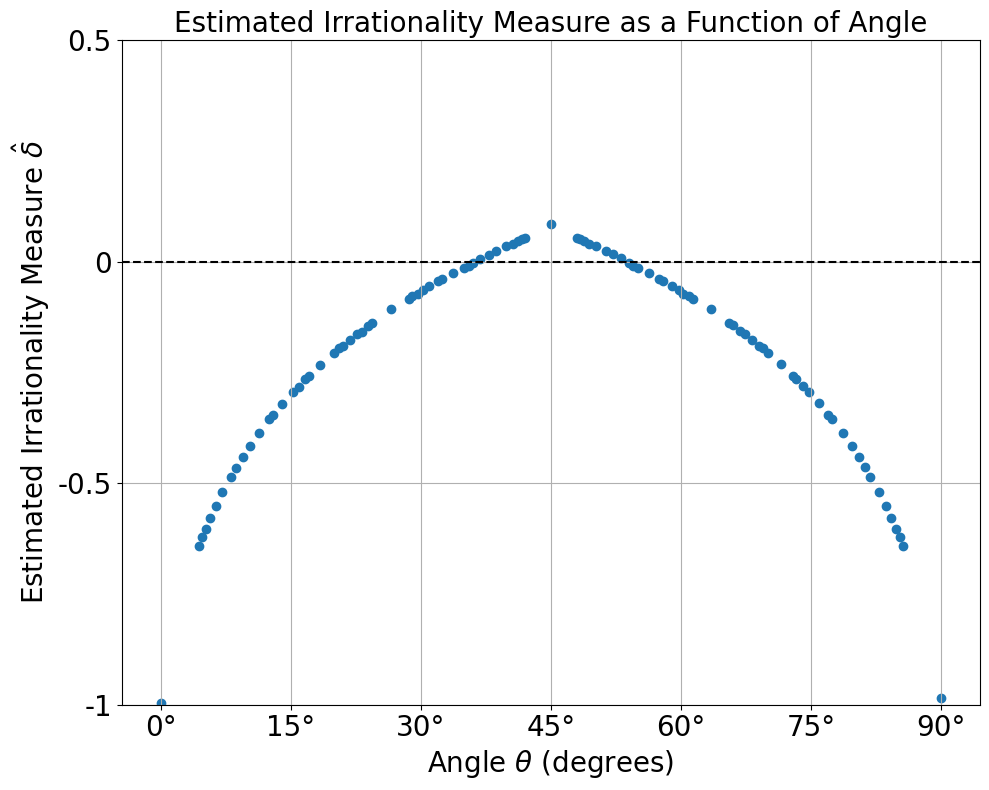}
    \end{subfigure}

    \vspace{0.25cm}

    \begin{subfigure}[b]{0.49\textwidth}
        \centering
        \includegraphics[width=\textwidth]{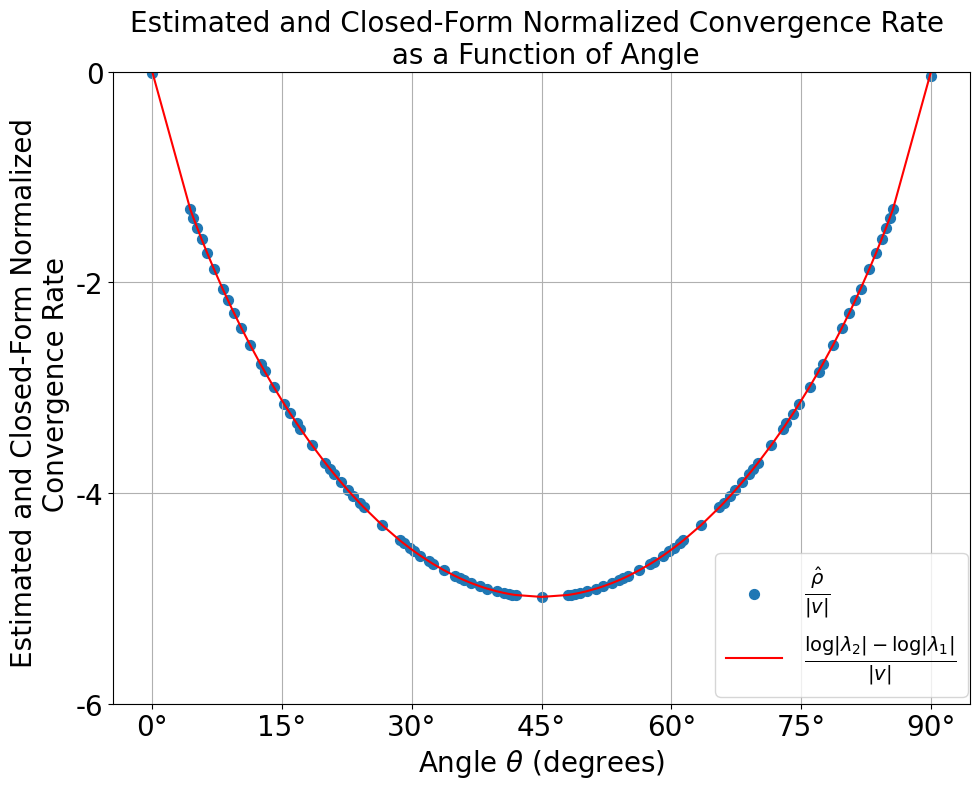}
    \end{subfigure}
    \hfill
    \begin{subfigure}[b]{0.49\textwidth}
        \centering
        \includegraphics[width=\textwidth]{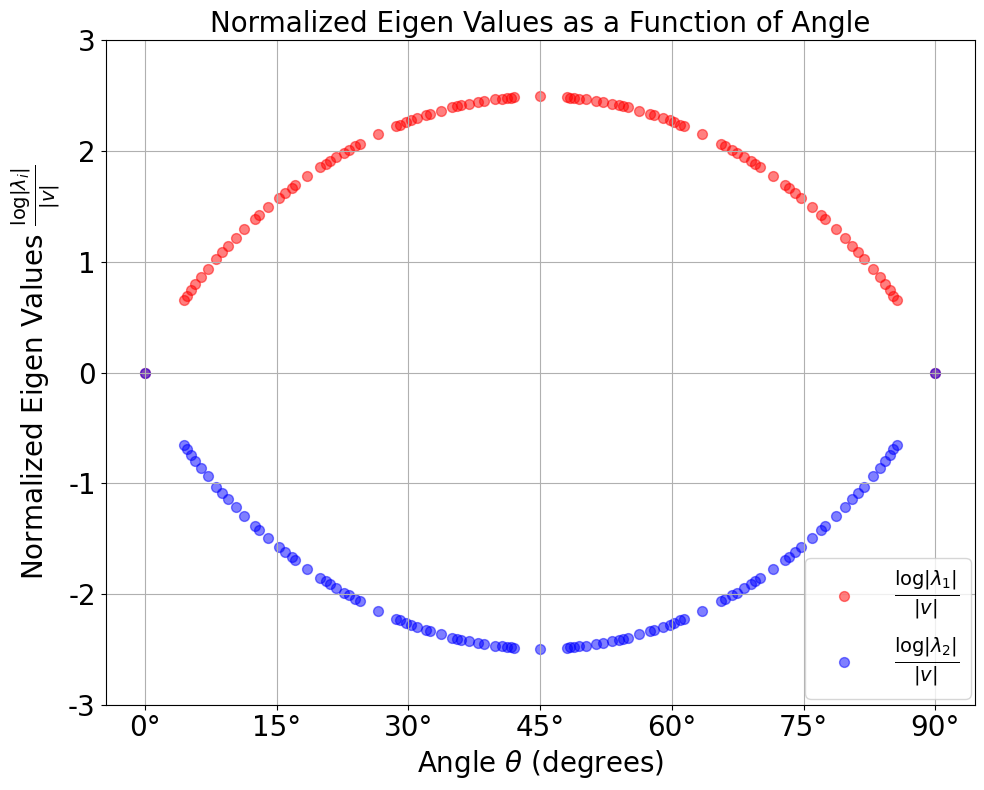}
    \end{subfigure}
    \caption{Asymptotic and arithmetic properties of the CMF from Example~\ref{ex:Zeta3CMF}, plotted versus the direction $\textbf{v}\in\N^2$. The horizontal axis in every panel is the angle of $\textbf{v}$ with the $x_1$-axis. \textbf{Top left:} the estimated limit $\hat{l}=\clL_{\textbf{x},\textbf{v}}^{\textbf{p},\textbf{q}}(1000)$; the red line marks $\zeta(3)$. \textbf{Top right:} the estimated irrationality measure $\hat{\delta}$; the black line marks the threshold $\delta=0$. \textbf{Bottom left:} the normalized convergence rate, with blue dots for the estimated values $\frac{\hat{\rho}(\textbf{v})}{\abs{\textbf{v}}}$ and a red curve for the closed-form value $\frac{\log\abs{\lambda_2(\textbf{v})}-\log\abs{\lambda_1(\textbf{v})}}{\abs{\textbf{v}}}$. \textbf{Bottom right:} the normalized eigenvalues $\frac{\log\abs{\lambda_i(\textbf{v})}}{\abs{\textbf{v}}}$.}
    \label{fig:zeta3panels}
\end{figure}
Let us note 3 observations from this experiment:
\begin{enumerate}
    \item The limit, which by Theorem~\ref{thm:cmf_rat_convergence} is given by $l(\textbf{v}) = \frac{\textbf{p}^{\mathsf T}B(\textbf{x},\textbf{v})e_1}{\textbf{q}^{\mathsf T}B(\textbf{x},\textbf{v})e_1}$, appears to be constantly equal to $\zeta(3)$. This would imply that for all $\textbf{v}$ in the first quadrant, the vector $(\textbf{p}-\zeta(3)\textbf{q})^{\mathsf T}$ is orthogonal to the first column of $B(\textbf{x},\textbf{v})$.
    \item \textbf{The irrationality measure $\delta(\textbf{v})$ appears to be a continuous function of the angle.} This is potentially useful for the search for irrationality-proving linear forms. This behavior seems to hold generally, as is evident in the other examined CMFs. It is formalized in Conjecture \ref{conj:contin_delta}
    \item In fact, one can infer from the graph a closed form for the irrationality measure $\delta(\textbf{v})$. This is given, together with additional details, in Appendix~\ref{app:zeta3_delta}.
\end{enumerate}
Next, we examine the CMF generated by the D-finite function $\pFq{2}{1}{2(x_1-x_2),2x_1+x_2}{x_1+2x_2}{-1}$. The generators of this CMF are given in Appendix~\ref{app:2F1_Subcmf}. It may also be viewed as an evaluation of a sub-CMF of $\clM^{{}_2F_1}$; see Appendix~\ref{app:Dual_Sub_CMF} for the definition of a sub-CMF.
We examine the sequences $\clL_{\textbf{x},\textbf{v}}^{\textbf{p},\textbf{q}}$ for this CMF with $\textbf{x}=(\frac{1}{3},-\frac{1}{3})$, $\textbf{p}=(1,0)$, $\textbf{q}=(0,1)$, and $\textbf{v}\in \set{(a,b)\in\Z^2 : \gcd(a,b)=1,\ \abs{(a,b)}<12}$. For each of the $264$ trajectories, we computed $\hat{l}(\textbf{v})$, $\hat{\delta}(\textbf{v})$, $\frac{\hat{\rho}(\textbf{v})}{\abs{\textbf{v}}}$, and $\frac{\log\abs{\lambda_i(\textbf{v})}}{\abs{\textbf{v}}}$ with $N=100$. These quantities are plotted as functions of the angle of $\textbf{v}$ with the $x_1$-axis in Figure~\ref{fig:2f1panels}.\\
\begin{figure}[htbp]
    \centering

    \begin{subfigure}[b]{0.49\textwidth}
        \centering
        \includegraphics[width=\textwidth]{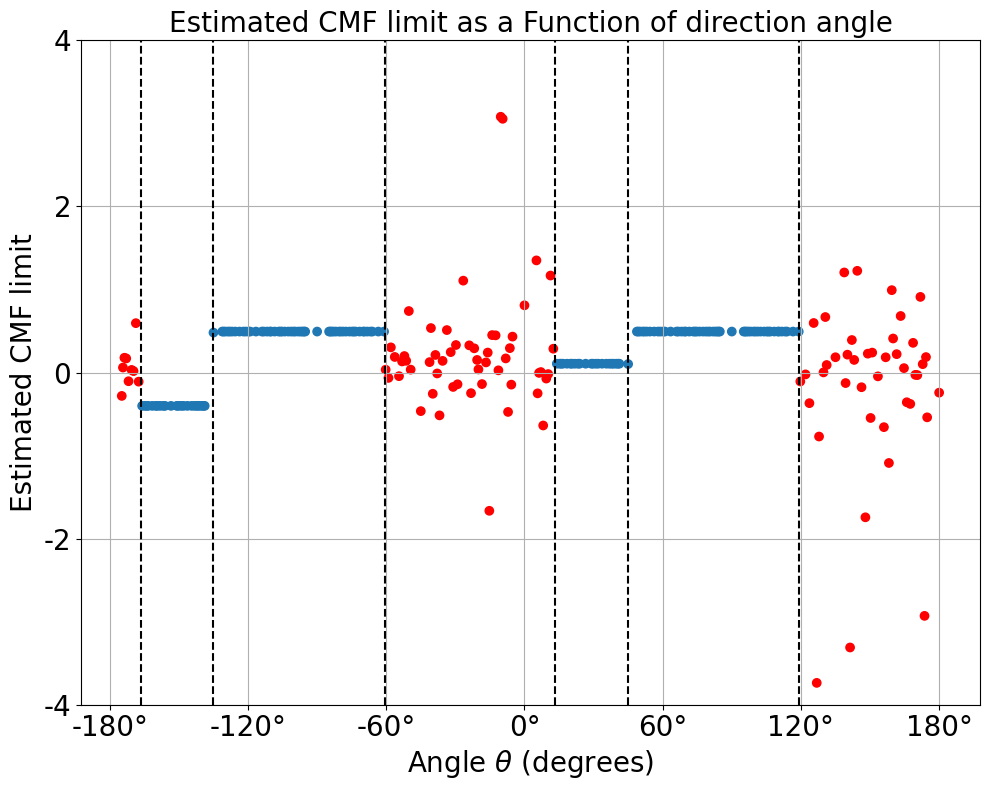}
    \end{subfigure}
    \hfill
    \begin{subfigure}[b]{0.49\textwidth}
        \centering
        \includegraphics[width=\textwidth]{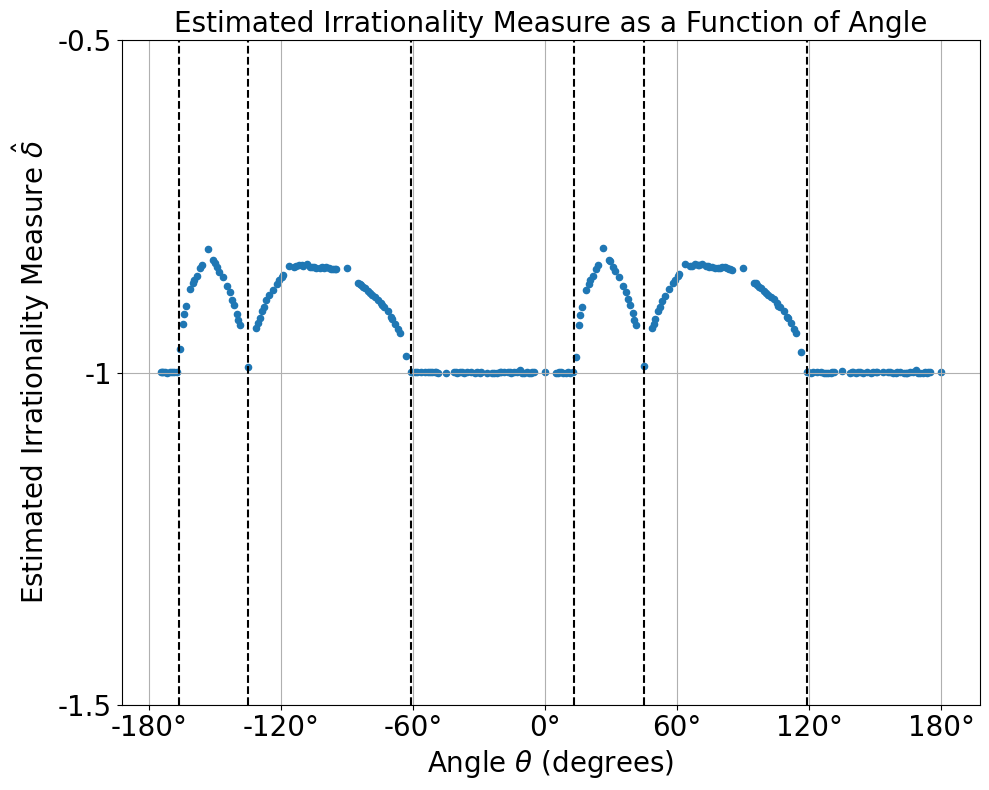}
    \end{subfigure}

    \vspace{0.25cm}

    \begin{subfigure}[b]{0.49\textwidth}
        \centering
        \includegraphics[width=\textwidth]{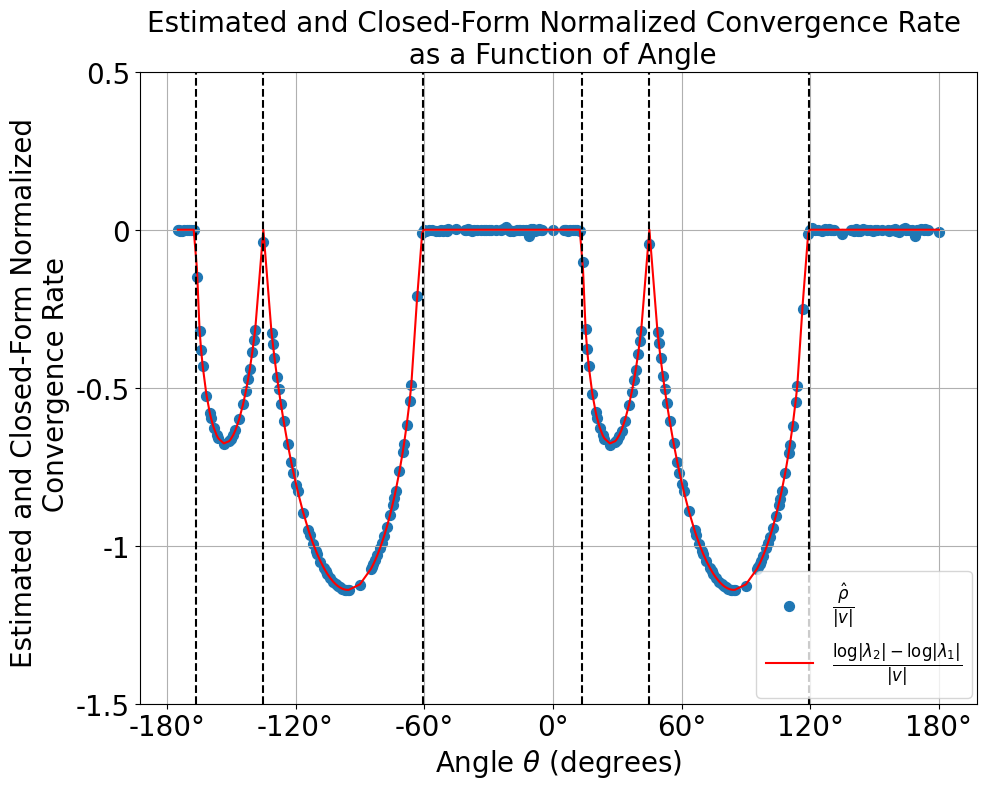}
    \end{subfigure}
    \hfill
    \begin{subfigure}[b]{0.49\textwidth}
        \centering
        \includegraphics[width=\textwidth]{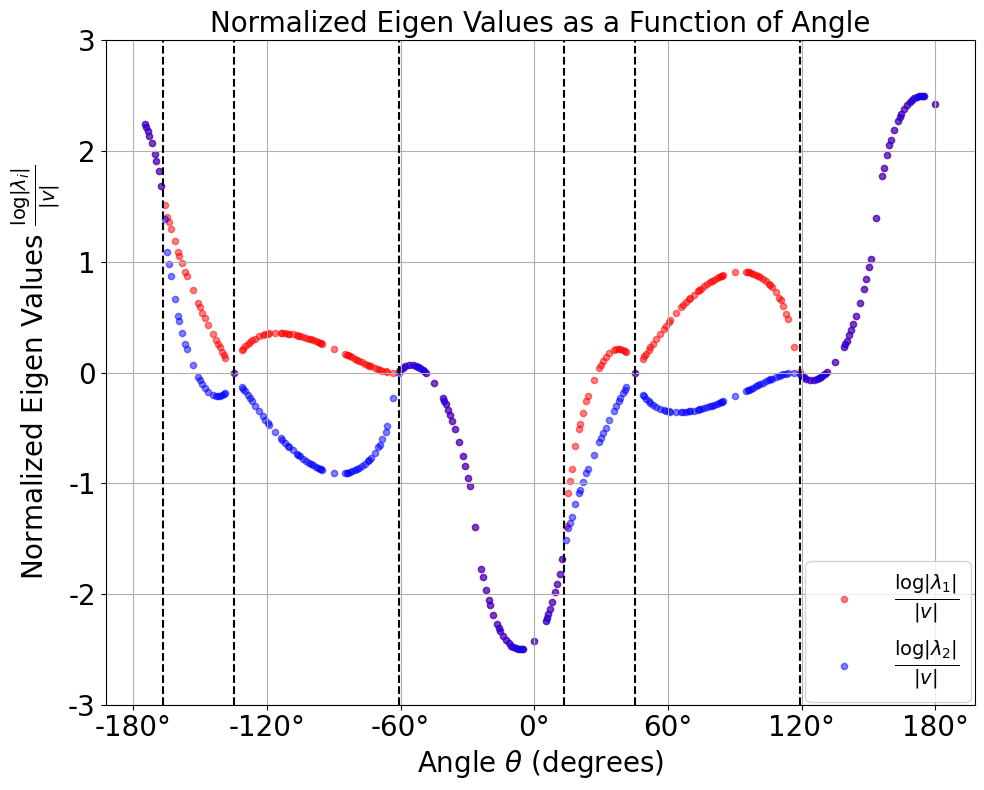}
    \end{subfigure}

    \caption{Asymptotic and arithmetic properties of the CMF from Appendix~\ref{app:2F1_Subcmf}, plotted versus the direction $\textbf{v}\in\Z^2$. In every panel, the horizontal axis is the angle of $\textbf{v}$ with the $x_1$-axis, and critical angles are indicated by black dashed lines. \textbf{Top left:} the estimated limit $\hat{l}=\clL_{\textbf{x},\textbf{v}}^{\textbf{p},\textbf{q}}(100)$; red points indicate directions for which the CMF ratio does not converge (see Figure~\ref{fig:2f1NonConverging}). \textbf{Top right:} the estimated irrationality measure $\hat{\delta}$. \textbf{Bottom left:} the normalized convergence rate, with blue dots for the estimated values $\frac{\hat{\rho}(\textbf{v})}{\abs{\textbf{v}}}$ and a red curve for the closed-form value $\frac{\log\abs{\lambda_2(\textbf{v})}-\log\abs{\lambda_1(\textbf{v})}}{\abs{\textbf{v}}}$. \textbf{Bottom right:} the normalized eigenvalues $\frac{\log\abs{\lambda_i(\textbf{v})}}{\abs{\textbf{v}}}$.}
    \label{fig:2f1panels}
\end{figure}
This experiment suggests a few further observations:
\begin{enumerate}
    \item The intervals where $\clL_{\textbf{x},\textbf{v}}^{\textbf{p},\textbf{q}}(n)$ does not converge, together with the points where the limits $\hat{l}(\textbf{v})$ jump, coincide with the points where the estimated convergence rate $\hat{\rho}(\textbf{v})$ is $0$, the estimated irrationality measure $\hat{\delta}(\textbf{v})$ is $-1$, and the eigenvalues have the same absolute value.
    \item On each interval where the convergence rate is nonzero, the limits $\hat{l}(\textbf{v})$ appear to be constant. This is reminiscent of the Stokes phenomenon and suggests that the space spanned by the first column of $B(\textbf{x},\textbf{v})$ remains constant on these intervals.
    \item Since we see the full $360\degree$ range here, the normalized convergence rate appears to be $2$-periodic. This can in fact be justified by examining the generators and using the fact that $r=2$.
    \item Once again, the irrationality measure appears to vary continuously.
\end{enumerate}
Finally, let us examine the CMF from Example~\ref{ex:Const3x3}, generated by
$$M_{1} =\begin{pmatrix}-9 & 2 & 2\\-38 & 11 & 4\\-24 & 4 & 7\end{pmatrix}\quad M_{2} = \begin{pmatrix}\frac{119}{6} & - \frac{7}{6} & - \frac{37}{6}\\\frac{343}{6} & - \frac{11}{6} & - \frac{119}{6}\\39 & - \frac{7}{3} & -12\end{pmatrix}$$
We examine the two sequences $\clL_{\textbf{x},\textbf{v}}^{\textbf{p}_1,\textbf{q}_1}$ and $\clL_{\textbf{x},\textbf{v}}^{\textbf{p}_2,\textbf{q}_2}$ for this CMF with $\textbf{x}=(0,0,0)$, $\textbf{p}_1=(1,0,0)$, $\textbf{q}_1=(0,1,0)$, $\textbf{p}_2=(-1,-1,2)$, $\textbf{q}_2=(3,-1,0)$, and $\textbf{v}\in \set{(a,b)\in\Z^2 : \gcd(a,b)=1,\ \abs{(a,b)}<12}$. For each of the $264$ trajectories, we computed $\hat{l}(\textbf{v})$, $\hat{\delta}(\textbf{v})$, $\frac{\hat{\rho}(\textbf{v})}{\abs{\textbf{v}}}$, and $\frac{\log\abs{\lambda_i(\textbf{v})}}{\abs{\textbf{v}}}$ with $N=200$. Writing $\hat{l}_1,\hat{\delta}_1,\hat{\rho}_1$ for the first sequence and $\hat{l}_2,\hat{\delta}_2,\hat{\rho}_2$ for the second, we plot these quantities as functions of the angle of $\textbf{v}$ with the $x_1$-axis in Figure~\ref{fig:C3x3panels}.
\begin{figure}[htbp]
    \centering

    \begin{subfigure}[b]{0.49\textwidth}
        \centering
        \includegraphics[width=\textwidth]{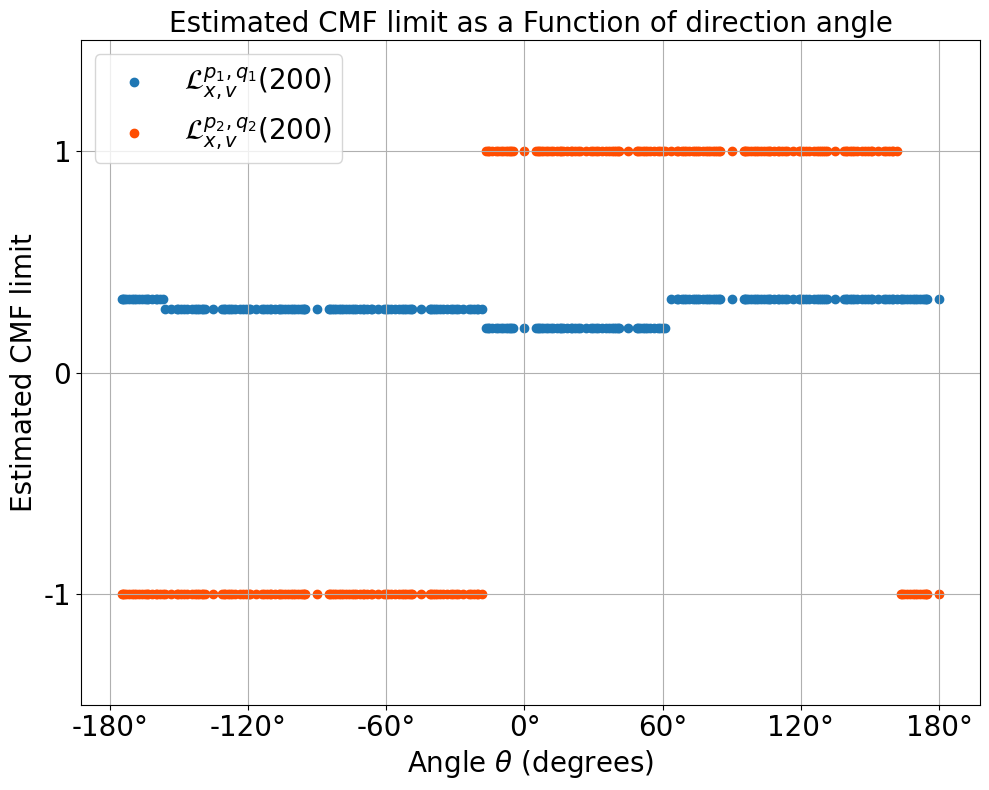}
    \end{subfigure}
    \hfill
    \begin{subfigure}[b]{0.49\textwidth}
        \centering
        \includegraphics[width=\textwidth]{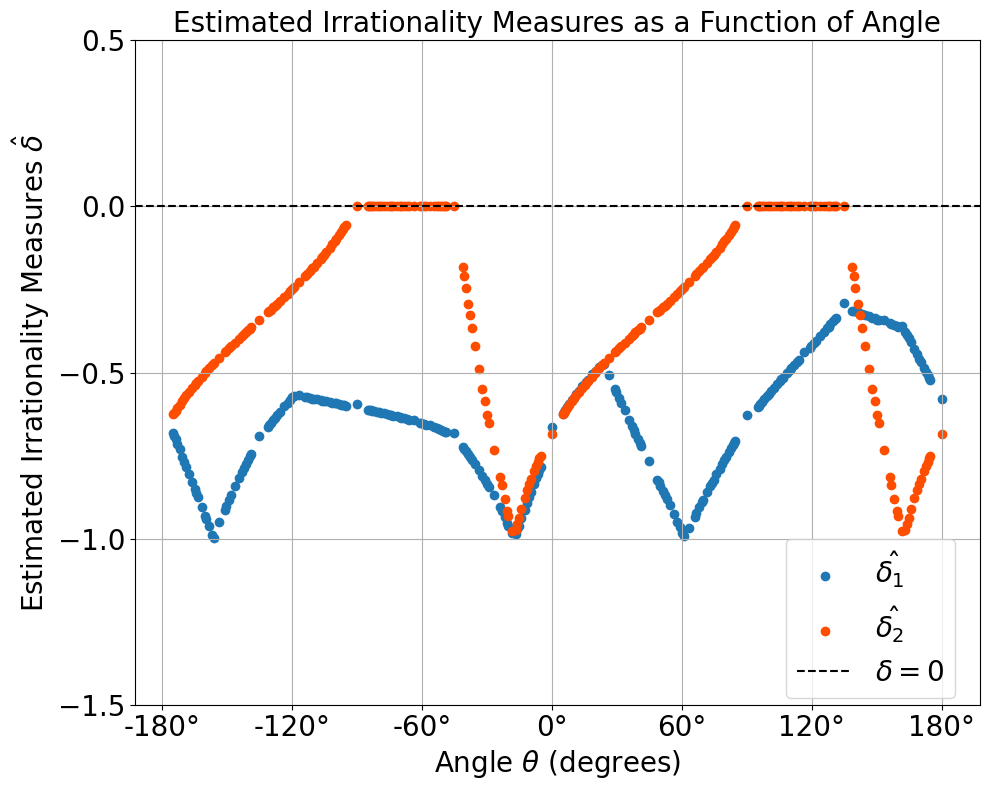}
    \end{subfigure}

    \vspace{0.25cm}

    \begin{subfigure}[b]{0.49\textwidth}
        \centering
        \includegraphics[width=\textwidth]{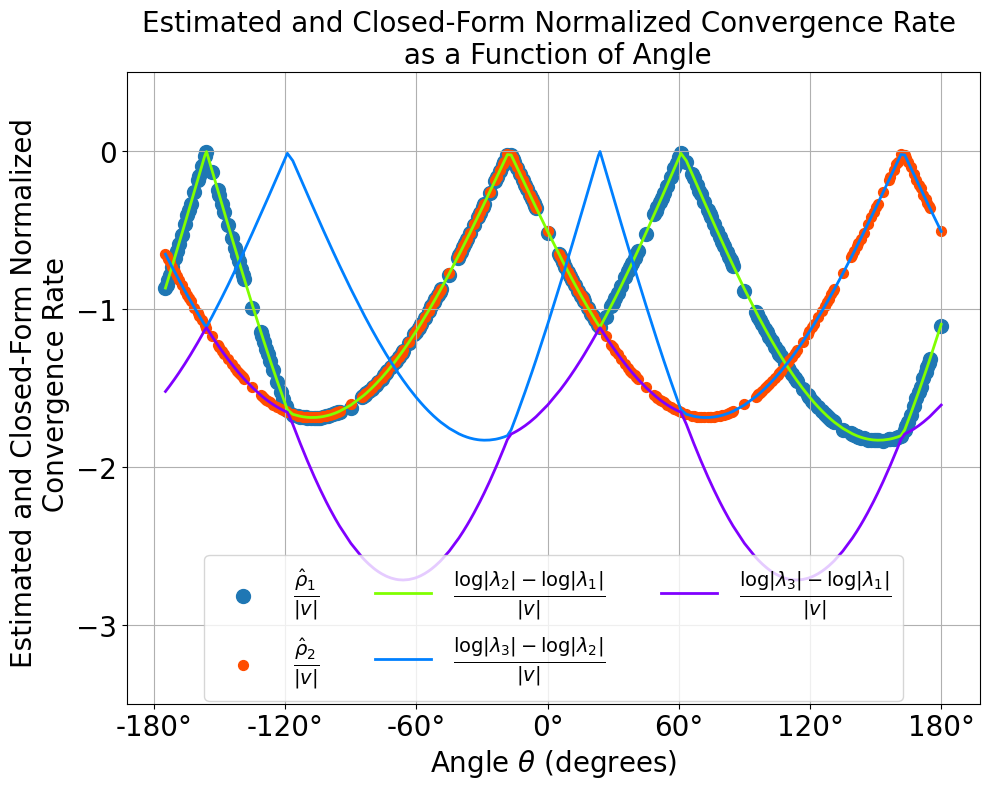}
    \end{subfigure}
    \hfill
    \begin{subfigure}[b]{0.49\textwidth}
        \centering
        \includegraphics[width=\textwidth]{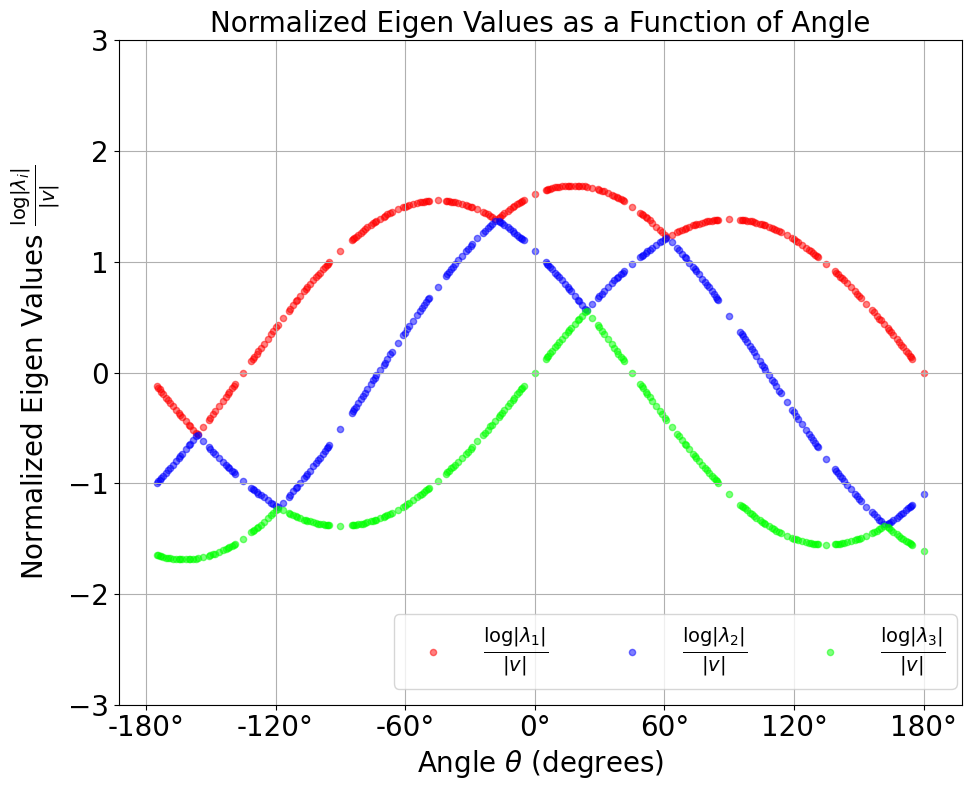}
    \end{subfigure}

    \caption{Asymptotic and arithmetic properties of the constant CMF in Example~\ref{ex:Const3x3}, plotted versus the direction $\textbf{v}\in\Z^2$. In each panel, the horizontal axis is the angle of $\textbf{v}$ with the $x_1$-axis. Blue denotes the parameters of $\clL_{\textbf{x},\textbf{v}}^{\textbf{p}_1,\textbf{q}_1}(n)$, and orange denotes those of $\clL_{\textbf{x},\textbf{v}}^{\textbf{p}_2,\textbf{q}_2}(n)$. The eigenvalues and closed-form convergence rates are independent of the choice of $\textbf{p},\textbf{q}$. \textbf{Top left:} the estimated limits $\hat{l}_1,\hat{l}_2$. \textbf{Top right:} the estimated irrationality measures $\hat{\delta}_1,\hat{\delta}_2$. \textbf{Bottom left:} the normalized convergence rates, with dots indicating the estimated values $\frac{\hat{\rho}_i(\textbf{v})}{\abs{\textbf{v}}}$ and curves the closed-form values $\frac{\log\abs{\lambda_j}(\textbf{v})-\log\abs{\lambda_i}(\textbf{v})}{\abs{\textbf{v}}}$. There are three possible closed-form values, depending on which eigenvalues appear in the asymptotic expansion of the sequence. \textbf{Bottom right:} the normalized eigenvalues $\frac{\log\abs{\lambda_i}(\textbf{v})}{\abs{\textbf{v}}}$.}
    \label{fig:C3x3panels}
\end{figure}
This experiment suggests:
\begin{enumerate}
    \item The normalized eigenvalues are sine functions of the angle of $\textbf{v}$ with the $x_1$-axis.
    \item The limits are constant on intervals where the convergence rate is nonzero.
    \item The irrationality measures appear to be continuous.
\end{enumerate}

We package the results of the experiments presented in Figures~\ref{fig:zeta3panels},~\ref{fig:2f1panels}, and~\ref{fig:C3x3panels}, together with others not shown here, into two main conjectures:
\begin{conjecture}\label{conj:ConstLimits} \textbf{(Local constancy of the asymptotic flag).}\\
Given a CMF $\clM$ and a direction $\textbf{v}$ satisfying the hypotheses of Proposition~\ref{prop:continuousEigenvalues}, let $V$ be an angular neighborhood of $\textbf{v}$ on which the eigenvalues remain distinct in modulus. Then, for any $\textbf{u}\in V$,
\[
B(\textbf{x},\textbf{u}) = B(\textbf{x},\textbf{v})U \qquad U \text{ is invertible upper-triangular}
\]
\end{conjecture}
Conjecture~\ref{conj:ConstLimits} implies the observed \textbf{local constancy of the limits}. It also implies continuity of the normalized convergence rate, since it yields the condition used in the proof of Proposition~\ref{prop:homog_and_contin_metrics}. One may think of this as an asymptotic analog of simultaneous diagonalization, but for the more delicate commutativity satisfied by CMF generators rather than ordinary commutativity.

We now turn to the irrationality measure.
\begin{conjecture}\label{conj:contin_delta} \textbf{(Continuous asymptotic height).}\\
Given a CMF $\clM$ and a direction $\textbf{v}$ satisfying the hypotheses of Theorem~\ref{thm:cmf_rat_convergence}, the irrationality-measure sequence, in the sense of Definition~\ref{def:irrationality_measure_of_seq}, of the CMF ratio $\clL_{\textbf{x},\textbf{v}}^{\textbf{p},\textbf{q}}(n)$ converges, and the resulting irrationality measure, denoted $\delta(\textbf{v})$, is continuous as a function of $\frac{\textbf{v}}{\abs{\textbf{v}}}$.
\end{conjecture}
Conjecture~\ref{conj:contin_delta} suggests that \textbf{optimization algorithms such as simulated annealing could be used to search for an irrationality-proving trajectory} $\textbf{v}$ in a CMF.
\clearpage
\section*{Acknowledgments}
This research received support through Schmidt Sciences, LLC.

The authors also wish to acknowledge the use of the RISC packages HolonomicFunctions package developed by Christoph Koutschan \cite{Koutschan2009} and the Asymptotics package by Manuel Kauers \cite{Kauers2011}. The first author is grateful to Frédéric Chyzak and Tali Monderer for their insightful discussions.

\section*{Declaration of generative AI and AI-assisted technologies in the manuscript preparation process}
During the preparation of this work the authors used ChatGPT 5.5 (OpenAI) for assistance in identifying potentially relevant literature, code generation, and brainstorming. After using this tool, the authors reviewed and edited the content as needed and take full responsibility for the content of the published article.
\printbibliography
\clearpage
\appendix
\section{Dual CMF, and Sub CMF}\label{app:Dual_Sub_CMF}
\begin{definition} \label{def:Dual_CMF}
    Let $\clM$ be a CMF of dimension $d$ rank $r$ over $K$. Then, the \textit{Dual CMF} of $\clM$, $\clM^*$ is a CMF of dimension $d$ rank $r$ over $K$ satisfying:
    $$\forall_{\textbf{v}\in\Z^d}\ \clM^*_\textbf{v} = (\clM_\textbf{v}^{-1})^\mathsf{T}$$
\end{definition}
\begin{definition}
    Let $\clM$ be a CMF of dimension $d$ rank $r$ over $K$. Let $\Lambda$ be a sub lattice of $\Z^d$, and let $(l_1,\dots,l_s)$ be a basis of $\Lambda$, and $(l_1,\dots,l_s,l_{s+1},\dots,l_{d})$ be a basis of $\Z^d$.\\
    Then, the matrices $\clM_{l_1},\dots,\clM_{l_s}$ can be rewritten as members of $\text{GL}_r(K(l_1,\dots,l_d))$, and in that form they generate a CMF of dimension $s$ rank $r$ over $K(l_{s+1},\dots,l_d)$. Such a CMF is called a \textit{Sub-CMF} of $\clM$.
\end{definition}
\begin{remark}
    Trajectory matrices $T_{\textbf{x},\textbf{v}}(n)$ are evaluations of the sub-CMF associated with $\Lambda = \set{k\textbf{v}:k\in\Z}$
\end{remark}
\section{CMF generated by $\pFq{2}{1}{2(x_1-x_2),2x_1+x_2}{x_1+2x_2}{-1}$} \label{app:2F1_Subcmf}
The CMF analyzed in Figure \ref{fig:2f1panels} is a CMF of dimension $2$ and rank $2$, generated by the following matrices:
\normalsize
$$M_1 = \begin{pmatrix}
    \frac{\left(x_{1} + 2 x_{2}\right) \left(3 x_{1} - 3 x_{2} + 2\right)}{\left(2 x_{1} + x_{2} + 1\right) \left(8 x_{1} - 8 x_{2} + 4\right)} & - \frac{\left(x_{1} + 2 x_{2}\right) \left(7 x_{1}^{2} + x_{1} \left(11 - 8 x_{2}\right) + x_{2}^{2} - 5 x_{2} + 4\right)}{\left(2 x_{1} + x_{2} + 1\right) \left(16 x_{1} - 16 x_{2} + 8\right)}\\\frac{\left(x_{1} + 2 x_{2}\right) \left(7 x_{1}^{2} - 8 x_{1} x_{2} + 5 x_{1} + x_{2}^{2} + x_{2}\right)}{\left(x_{1} - x_{2}\right) \left(2 x_{1} + x_{2}\right) \left(2 x_{1} + x_{2} + 1\right) \left(16 x_{1} - 16 x_{2} + 8\right)} & - \frac{\left(x_{1} + 2 x_{2}\right) \left(11 x_{1}^{3} + x_{1}^{2} \left(3 x_{2} + 14\right) + x_{1} \left(- 3 x_{2}^{2} + 20 x_{2} + 1\right) - 11 x_{2}^{3} + 2 x_{2}^{2} + 11 x_{2} - 2\right)}{\left(x_{1} - x_{2}\right) \left(2 x_{1} + x_{2}\right) \left(2 x_{1} + x_{2} + 1\right) \left(32 x_{1} - 32 x_{2} + 16\right)}
\end{pmatrix}$$
\normalsize
$$M_2 = \begin{pmatrix}
P_{1,1} & P_{1,2}\\
P_{2,1} & P_{2,2}
\end{pmatrix}$$
Where:
$$P_{1,1} = \frac{\left(x_{1} + 2 x_{2}\right) \left(x_{1} + 2 x_{2} + 1\right) \left(7 x_{1}^{2} - x_{1} \left(50 x_{2} + 19\right) + 79 x_{2}^{2} + 67 x_{2} + 10\right)}{\left(x_{1} - 4 x_{2}\right) \left(x_{1} - 4 x_{2} - 1\right) \left(x_{1}^{2} - x_{1} \left(8 x_{2} + 5\right) + 16 x_{2}^{2} + 20 x_{2} + 6\right)}$$
$$P_{1,2} =  - \frac{\left(2 x_{1} + 4 x_{2}\right) \left(x_{1} + 2 x_{2} + 1\right) \left(6 x_{1}^{3} - 3 x_{1}^{2} \left(15 x_{2} + 7\right) + x_{1} \left(90 x_{2}^{2} + 96 x_{2} + 22\right) - 51 x_{2}^{3} - 93 x_{2}^{2} - 49 x_{2} - 7\right)}{\left(x_{1} - 4 x_{2}\right) \left(x_{1} - 4 x_{2} - 1\right) \left(x_{1}^{2} - x_{1} \left(8 x_{2} + 5\right) + 16 x_{2}^{2} + 20 x_{2} + 6\right)}$$
$$P_{2,1} = \frac{\left(6 x_{1} + 12 x_{2}\right) \left(x_{1} + 2 x_{2} + 1\right) \left(2 x_{1}^{3} - 3 x_{1}^{2} \left(5 x_{2} + 2\right) + x_{1} \left(30 x_{2}^{2} + 27 x_{2} + 4\right) - x_{2} \left(17 x_{2}^{2} + 21 x_{2} + 6\right)\right)}{\left(x_{1} - 4 x_{2}\right) \left(x_{1} - x_{2}\right) \left(2 x_{1} + x_{2}\right) \left(x_{1} - 4 x_{2} - 1\right) \left(x_{1}^{2} - x_{1} \left(8 x_{2} + 5\right) + 16 x_{2}^{2} + 20 x_{2} + 6\right)} $$
$$P_{2,2} = - \frac{4 \left(x_{1} + 2 x_{2}\right) \left(- x_{1} + x_{2} + 1\right)^{2} \left(x_{1} + 2 x_{2} + 1\right) \left(4 x_{1}^{2} - 2 x_{1} \left(10 x_{2} + 3\right) + x_{2} \left(7 x_{2} + 3\right)\right)}{\left(x_{1} - 4 x_{2}\right) \left(x_{1} - x_{2}\right) \left(2 x_{1} + x_{2}\right) \left(x_{1} - 4 x_{2} - 1\right) \left(x_{1}^{2} - x_{1} \left(8 x_{2} + 5\right) + 16 x_{2}^{2} + 20 x_{2} + 6\right)} $$
\section{Irrationality Measure in $\zeta(3)$ CMF}\label{app:zeta3_delta}
As mentioned in Section~\ref{sec:Experiments}, one can find conjectural closed forms for the irrationality measure $\delta(\textbf{v})$. We now present such a closed form for the $\zeta(3)$ CMF from Example~\ref{ex:Zeta3CMF}. First, compute the normalized eigenvalues symbolically:
\[
\frac{\log\abs{\lambda_1(\textbf{v})}}{\abs{\textbf{v}}} = 2\sin (\theta) \ln \abs{\tan \left(\frac{\theta}{2}\right)}+\cos (\theta) \ln \abs{ 1-\frac{2}{\sin (\theta)+1}} \quad , \quad \frac{\log\abs{\lambda_2(\textbf{v})}}{\abs{\textbf{v}}} = -\frac{\log\abs{\lambda_1(\textbf{v})}}{\abs{\textbf{v}}}
\]
Here $\theta$ is the angle of $\textbf{v}$ with the $x_1$-axis. With this notation, the following function matches the numerical observations:
\begin{equation}\label{eq:zeta3_delta_closed}
    \delta(\textbf{v}) = -1-\frac{\log\abs{\lambda_2(\textbf{v})}-\log\abs{\lambda_1(\textbf{v})}}{\log\abs{\lambda_1(\textbf{v})}-3\max\{v_1,v_2\}} \qquad \text{(conjecture)}
\end{equation}
where $\textbf{v}=(v_1,v_2)$ and $\max\{v_1,v_2\}$ denotes the maximum of its two entries. While not proven yet, this can be justified somewhat by noting that the entries of $\clM_{\textbf{v}}(\textbf{x})$ are integers when multiplied by ${\operatorname{lcm}(1,\dots,\max\{v_1,v_2\})}^3$, making~\eqref{eq:zeta3_delta_closed} a lower bound. The tightness of this lower bound, as well as the tightness of similar lower bounds generated by Rhin, Viola, and others~\cite{RhinViola1996,RhinViola2001,Fischler2003}, is left for future work.
\end{document}

%% file: preamble.tex
\usepackage{amsmath}
\usepackage{amsfonts}
\usepackage{gensymb}

\def\<={\leq}
\def\>={\geq}

\def\Z{\mathbb{Z}}
\def\R{\mathbb{R}}
\def\Q{\mathbb{Q}}
\def\C{\mathbb{C}}
\def\N{\mathbb{N}}
\def\A{\mathbb{A}}
\def\P{\mathbb{P}}
\newcommand{\Id}{\operatorname{Id}}
\newcommand{\GL}{\operatorname{GL}}
\newcommand{\diag}{\operatorname{diag}}
\newcommand{\set}[1]{\left\{ #1 \right\}}
\newcommand{\abs}[1]{\left| #1 \right|}

\newmuskip\pFqmuskip

\newcommand*\pFq[6][8]{%
  \begingroup 
  \pFqmuskip=#1mu\relax
  \mathcode`\,=\string"8000
  \begingroup\lccode`\~=`\,
  \lowercase{\endgroup\let~}\pFqcomma
  {}_{#2}F_{#3}{\left[\genfrac..{0pt}{}{#4}{#5};#6\right]}%
  \endgroup
}
\newcommand{\pFqcomma}{\mskip\pFqmuskip}

\DeclareUnicodeCharacter{05D0}{\hebalef}
\DeclareUnicodeCharacter{05D1}{\hebbet}
\DeclareUnicodeCharacter{05D2}{\hebgimel}
\DeclareUnicodeCharacter{05D3}{\hebdalet}
\DeclareUnicodeCharacter{05D4}{\hebhe}
\DeclareUnicodeCharacter{05D5}{\hebvav}
\DeclareUnicodeCharacter{05D6}{\hebzayin}
\DeclareUnicodeCharacter{05D7}{\hebhet}
\DeclareUnicodeCharacter{05D8}{\hebtet}
\DeclareUnicodeCharacter{05D9}{\hebyod}
\DeclareUnicodeCharacter{05DA}{\hebfinalkaf}
\DeclareUnicodeCharacter{05DB}{\hebkaf}
\DeclareUnicodeCharacter{05DC}{\heblamed}
\DeclareUnicodeCharacter{05DD}{\hebfinalmem}
\DeclareUnicodeCharacter{05DE}{\hebmem}
\DeclareUnicodeCharacter{05DF}{\hebfinalnun}
\DeclareUnicodeCharacter{05E0}{\hebnun}
\DeclareUnicodeCharacter{05E1}{\hebsamekh}
\DeclareUnicodeCharacter{05E2}{\hebayin}
\DeclareUnicodeCharacter{05E3}{\hebfinalpe}
\DeclareUnicodeCharacter{05E4}{\hebpe}
\DeclareUnicodeCharacter{05E5}{\hebfinaltsadi}
\DeclareUnicodeCharacter{05E6}{\hebtsadi}
\DeclareUnicodeCharacter{05E7}{\hebqof}
\DeclareUnicodeCharacter{05E8}{\hebresh}
\DeclareUnicodeCharacter{05E9}{\hebshin}
\DeclareUnicodeCharacter{05EA}{\hebtav}

\newcommand{\clL}{\mathcal{L}}
\newcommand{\clM}{\mathcal{M}}

\newcommand{\clZ}{\mathcal{Z}}  